\documentclass[a4paper, 10pt, reqno]{amsart}

\usepackage[utf8]{inputenc}

\usepackage{graphicx}
\usepackage{amsmath}
\usepackage{amsfonts}
\usepackage{amssymb}
\usepackage{amsthm}
\usepackage[pagewise]{lineno}
\usepackage[doublespacing]{setspace}

\theoremstyle{plain}
\newtheorem{theorem}{Theorem}

\numberwithin{theorem}{section}

\usepackage{hyperref}
\hypersetup{
colorlinks=true,
linkcolor=blue,
citecolor=red,
filecolor=magenta,
urlcolor=cyan,
pdftitle={},
pdfauthor={},
bookmarks=true,
}

\begin{document}

\author{Abhishek Das}
\title[Vanishing viscosity limits]{Explicit structure of the vanishing viscosity limits with initial data consisting of $\delta$-distributions starting from two point sources}

\address{Abhishek Das 
\newline 
National Institute of Science Education and Research Bhubaneswar,
P.O. Jatni, Khurda 752050, Odisha, India
\newline
An OCC of Homi Bhabha National Institute,
Training School Complex,
Anushaktinagar, Mumbai 400094, India
}

\email{abhishek.das@niser.ac.in, ad16.1992@gmail.com}



\maketitle

\begin{abstract}

In this article, we consider the one-dimensional zero-pressure gas dynamics system 

\[
u_t + \left( {u^2}/{2} \right)_x = 0,\ \rho_t + (\rho u)_x = 0
\]
in the upper-half plane with a linear combination of two $\delta$-distributions

\[ 
u|_{t=0} = u_a\ \delta_{x=a} + u_b\ \delta_{x=b},\ \rho|_{t=0} = \rho_c\ \delta_{x=c} + \rho_d\ \delta_{x=d}
\]
as initial data. Here $a$, $b$, $c$, $d$ are distinct points on the real line ordered as $a < c < b < d$. Our objective is to provide a detailed analysis of the structure of the vanishing viscosity limits of this system utilizing the corresponding modified adhesion model

\[
u^\epsilon_t + \left({(u^\epsilon)^2}/{2} \right)_x =\frac{\epsilon}{2} u^\epsilon_{xx},\ \rho^\epsilon_t + (\rho^\epsilon u^\epsilon)_x = \frac{\epsilon}{2} \rho^\epsilon_{xx}.
\] 
For this purpose, we extensively use the various asymptotic properties of the function erfc$: z \longmapsto \int_{z}^{\infty} e^{-s^2}\ ds$ along with suitable Hopf-Cole transformations.
\end{abstract}

\section{Introduction}

We consider the initial value problem for the system

\begin{equation}
u_t + \left( {u^2}/{2} \right)_x = 0,\ \rho_t + (\rho u)_x = 0,\ \left( x , t \right) \in \textbf{R}^1 \times \left( 0 , \infty \right),
\label{intro-1}
\end{equation}
under the assumption that the initial data consists of a linear combination of $\delta$-distributions, more precisely,

\begin{equation}
u|_{t=0} = u_a\ \delta_{x=a} + u_b\ \delta_{x=b},\ \rho|_{t=0} = \rho_c\ \delta_{x=c} + \rho_d\ \delta_{x=d},
\label{intro-2}
\end{equation}
where $a,b,c,d$ are points on the real line positioned according to the inequalities $a < c < b < d$ and $u_a$, $u_b$, $\rho_c$, $\rho_d \in \textbf{R}^1$ are real constants. 

The system \eqref{intro-1} is an important system of partial differential equations called the zero pressure gas dynamics system. It finds applications in cosmology and is closely related to the Zeldovich approximation (\cite{z1}). This system describes the evolution of matter
in the expansion of universe as cold dust moving under the effect of gravity alone. Here $u$ and $\rho$ respectively denote the velocity and the density of the particles, $x \in {\textbf{R}}^1$ denotes the space variable and $t>0$ denotes time. 

The first equation in \eqref{intro-1} is the Burgers equation. In general, even for smooth initial data, we cannot expect to find solutions of this problem in the class of smooth functions. Therefore, we have to look for weak solutions satisfying the given system \eqref{intro-1} in the sense of distributions, namely

\begin{small}
\[
\begin{aligned}
\int_{0}^{\infty} \int_{-\infty}^{\infty} \left( u\ \phi_t + f(u)\ \phi_x \right)\ dx dt + \int_{-\infty}^{\infty} u(x,0)\ \phi (x,0)\ dx = 0,\ \phi \in C_c^\infty \left( \textbf{R}^1 \times [ 0 , \infty ) \right).
\end{aligned} 
\]
\end{small}
The initial value problem for the Burgers equation was studied by Hopf \cite{h1} using the method of vanishing viscosity. He considered the problem

\[
\begin{aligned}
u_t + \left( \frac{u^2}{2} \right)_x &= \mu u_{xx},
\\
u|_{t=0} &= u_0
\end{aligned}
\]
for $\mu > 0$ with locally integrable initial data $u_0$ under the additional assumption that

\[
\lim_{|x| \rightarrow \infty} \frac{\int_{0}^{x} u_0 \left( \xi \right)\ d\xi}{x^2} = 0.
\] 
This resulted in an explicit formula

\[
\begin{aligned}
u(x,t) = \frac{\int_{-\infty}^{\infty}\ \frac{x-y}{t}\ e^{- \frac{F(x,y,t)}{2 \mu}}}{\int_{-\infty}^{\infty}\ e^{- \frac{F(x,y,t)}{2 \mu}}},
\end{aligned}
\]
where 

\[
\begin{aligned}
F(x,y,t) = \frac{(x-y)^2}{2t} + \int_{0}^{y}\ u_0 \left( \xi \right)\ d\xi,
\end{aligned}
\]
and the subsequent passage to the limit as $\mu \rightarrow 0$ was discussed in detail.

Lax \cite{lax1} extended the work of Hopf \cite{h1} to general conservation laws of the form

\[
u_t + \left( f(u) \right)_x = 0,
\]
under bounded measurable initial data $u|_{t=0} = u_0$, where the flux function $f \in C^2$ was assumed to be strictly convex and have superlinear growth at infinity:

\[
\begin{aligned}
f'' > 0,\ \lim_{y \rightarrow \infty} \frac{f(y)}{|y|} = \infty.
\end{aligned}
\]
A weak solution for this initial value problem was obtained through the following minimization problem:

For each fixed $t>0$, except possibly for countably many $x \in \textbf{R}^1$, there exists a unique minimizer $y(x,t)$ for

\[
\begin{aligned}
\theta \left( x,y,t \right) := f^* \left( \frac{x-y}{t} \right) + \int_{0}^{y}\ u_0 \left( \xi \right) d\xi,
\end{aligned}
\]
where $f^* \left( z \right) = \max_{p \in \textbf{R}^1} \{ pz - f(p) \}$ denotes the convex conjugate of $f$. 

A weak solution to the above initial value problem was then given by

\[
\begin{aligned}
\overline{u} (x,t) = \left( f^* \right)^{-1} \left( \frac{x-y(x,t)}{t} \right),
\end{aligned}
\]
which was defined for a.e. $x \in \textbf{R}^1$.

For scalar conservation laws with bounded nonnegative measures as initial data, there are many published works (see Bertsch et. al. \cite{ber-sm-tr-tes-20}, Demengel and Serre \cite{dem-serre-91}, Liu and Pierre \cite{liu-pie-84} and the references given there). These works are concerned with the study of scalar conservation laws of the form

\[
\begin{aligned}
u_{t}+ ( \phi(u) )_x = 0,
\end{aligned}
\]
with the initial data at $t=0$, namely

\[
\begin{aligned}
u|_{t=0} = u_0,
\end{aligned}
\]
being a bounded measure. The existence and uniqueness of solutions depend on this initial data, on whether the solution is nonnegative or not and on whether the flux $\phi$ is odd or convex. It is natural to expect a smoothing effect on the solution because of the strong nonlinearity of the flux $\phi$. 

The vanishing viscosity method has been previously applied by Joseph \cite{ktj-93} for the system \eqref{intro-1} under Riemann-type initial data, namely

\[
\begin{aligned}
u^\epsilon_t + \left({(u^\epsilon)^2}/{2} \right)_x &=\frac{\epsilon}{2} u^\epsilon_{xx},\ \rho^\epsilon_t + (\rho^\epsilon u^\epsilon)_x = \frac{\epsilon}{2} \rho^\epsilon_{xx},
\\
\left( u^\epsilon(x,0) , \rho^\epsilon(x,0) \right) &= \begin{cases}
\left( u_{{}_{L}} , \rho_{{}_{L}} \right),\ &x<0,
\\
\left( u_{{}_{R}} , \rho_{{}_{R}} \right),\ &x>0.
\end{cases}
\end{aligned} 
\]
The given problem was subsequently reduced to the new system

\[
\begin{aligned}
U_t^\epsilon + \frac{\left( U_x^\epsilon \right)^2}{2} = \frac{\epsilon}{2} U_{xx}^\epsilon,\ R_t^\epsilon + R_x^\epsilon U_x^\epsilon = \frac{\epsilon}{2} R_{xx}^\epsilon,
\end{aligned}
\] 
and the new initial data was

\[
\left( U^\epsilon(x,0) , R^\epsilon(x,0) \right) = \begin{cases}
\left( u_{{}_{L}} x , \rho_{{}_{L}} x \right),\ &x<0,
\\
\left( u_{{}_{R}} x , \rho_{{}_{R}} x \right),\ &x>0.
\end{cases}
\] 
At this point, it was observed that 
\[
\left( \overline{u}^\epsilon , \overline{R}^\epsilon \right) = \left( U_x^\epsilon, R_x^\epsilon \right)
\]
solved \eqref{intro-1} under the initially prescribed Riemann-type data. The first step after this reduction was to linearise the first equation in \eqref{intro-1} by the Hopf-Cole transformation

\[
V^\epsilon = e^{- \frac{U^\epsilon}{\epsilon}}.
\]
This resulted in the consideration of the linear problem

\[
\begin{aligned}
V_t^\epsilon &= V_{xx}^\epsilon,
\\
V^\epsilon (x,0) &= \begin{cases}
e^{- \frac{u_{{}_{L}} x}{\epsilon}},\ &x < 0,
\\
e^{- \frac{u_{{}_{R}} x}{\epsilon}},\ &x > 0.
\end{cases}
\end{aligned}
\]
The second equation in \eqref{intro-1} involving $\rho$ was linearised by a modified Hopf-Cole transformation

\[
S^\epsilon = R^\epsilon\ e^{- \frac{U^\epsilon}{\epsilon}}.
\]
The corresponding linear problem would be

\[
\begin{aligned}
S_t^\epsilon &= S_{xx}^\epsilon,
\\
S^\epsilon (x,0) &= \begin{cases}
\rho_{{}_{L}} x\ e^{- \frac{u_{{}_{L}} x}{\epsilon}},\ &x < 0,
\\
\rho_{{}_{R}} x\ e^{- \frac{u_{{}_{R}} x}{\epsilon}},\ &x > 0.
\end{cases}
\end{aligned}
\]
These two linear problems could then be explicitly solved using the heat kernel and the $\left( u^\epsilon , \rho^\epsilon \right)$ were recovered using the relations

\[
\left( u^\epsilon , \rho^\epsilon \right) = \left( - \epsilon \cdot \frac{V_x^\epsilon}{V^\epsilon} , \left( \frac{S^\epsilon}{V^\epsilon} \right)_x \right).
\]
The above strategy has be used in \cite{das-ktj-submitted} for the problem
\[
\begin{aligned}
u_t + \left( {u^2}/{2} \right)_x &= 0,\ \rho_t + (\rho u)_x = 0,
\\
u|_{t=0} &= u_a\ \delta_{x=a} + u_b\ \delta_{x=b},\ \rho|_{t=0} = \rho_a\ \delta_{x=a} + \rho_b\ \delta_{x=b},
\end{aligned}
\]
where $-\infty < a < b < \infty$ and $u_a$, $u_b$, $\rho_a$, $\rho_b$ are real constants.
\section{Computation of the vanishing viscosity limits}

In this section, we generalize the ideas developed in \cite{das-ktj-submitted} to study the structure of the vanishing viscosity limits for the system

\begin{equation}
u_t + \left( {u^2}/{2} \right)_x = 0,\ \rho_t + (\rho u)_x = 0,\ \left( x , t \right) \in \textbf{R}^1 \times \left( 0 , \infty \right)
\label{comms-1}
\end{equation}
subject to the initial data

\begin{equation}
u|_{t=0} = u_a\ \delta_{x=a} + u_b\ \delta_{x=b},\ \rho|_{t=0} = \rho_c\ \delta_{x=c} + \rho_d\ \delta_{x=d},
\label{comms-2}
\end{equation}
where $-\infty < a < c < b < d < \infty$ and $u_a$, $u_b$, $\rho_c$, $\rho_d$ are real constants. We use the modified adhesion model associated to \eqref{comms-1}-\eqref{comms-2}, namely 

\[
\begin{aligned}
u^\epsilon_t + \left({(u^\epsilon)^2}/{2} \right)_x &=\frac{\epsilon}{2} u^\epsilon_{xx},\ \rho^\epsilon_t + (\rho^\epsilon u^\epsilon)_x = \frac{\epsilon}{2} \rho^\epsilon_{xx},
\\
u^\epsilon|_{t=0} &= u_a\ \delta_{x=a} + u_b\ \delta_{x=b},\ \rho^\epsilon|_{t=0} = \rho_c\ \delta_{x=c} + \rho_d\ \delta_{x=d}.
\end{aligned}
\]

In order to compute the vanishing viscosity limits, we need to consider different cases depending on the relative signs of $u_a$ and $u_b$. We prove the following theorem.

\begin{theorem}

Suppose $a$, $b$, $c$, $d$ are points on the real line ordered according to the inequalities $-\infty < a < c < b < d < \infty$. Given real constants $u_a$, $u_b$, $\rho_c$ and $\rho_d$, let us consider the one-dimensional zero-pressure gas dynamics system

\[
u_t + \left( {u^2}/{2} \right)_x = 0,\ \rho_t + (\rho u)_x = 0,\ \left( x , t \right) \in {\textnormal{\textbf{R}}}^1 \times \left( 0 , \infty \right)
\] 
under the initial data

\[
u|_{t=0} = u_a\ \delta_{x=a} + u_b\ \delta_{x=b},\ \rho|_{t=0} = \rho_c\ \delta_{x=c} + \rho_d\ \delta_{x=d}.
\] 
Suppose $u^\epsilon$, $\rho^\epsilon$ are approximate solutions of the system

\[
\begin{aligned}
u^\epsilon_t + \left( \frac{(u^\epsilon)^2}{2} \right)_x &= \frac{\epsilon}{2} u^\epsilon_{xx},\ \rho^\epsilon_t + ( \rho^\epsilon u^\epsilon )_x = \frac{\epsilon}{2} \rho^\epsilon_{xx},
\\
u^\epsilon|_{t=0} &= u_a\ \delta_{x=a} + u_b\ \delta_{x=b},\ \rho^\epsilon|_{t=0} = \rho_c\ \delta_{x=c} + \rho_d\ \delta_{x=d}.
\end{aligned}
\]
Depending on the relative signs of $u_a$ and $u_b$, we may describe the precise asymptotic behavior of the vanishing viscosity limits $\lim_{\epsilon \rightarrow 0} \left( u^\epsilon , \rho^\epsilon \right)$ as follows:

\textbf{Case 1.} $u_a < 0$, $u_b > 0$
\\	
Consider the curves defined by
	
\[
\begin{aligned}
\gamma_{{}_a} (t) 
	&:= a - \sqrt{- 2 u_a t},\ \gamma_{{}_b} (t) := b + \sqrt{2 u_b t},\ \gamma_{{}_{c}} (t) := c,\ t \geq 0,
	\\
\gamma_{{}_d} (t) 
	&:= \begin{cases}
d,\ &0 \leq t \leq \frac{(d-b)^2}{2 u_b},
	\\
b + \sqrt{2 u_b t},\ &t \geq \frac{(d-b)^2}{2 u_b}. 			
\end{cases}	    
\end{aligned}
\]
Then the velocity component $u$ and the density component $\rho$ can be explicitly described as follows:

\[
\begin{aligned}
u(x,t) &= \begin{cases}
0,\ &x \in ( -\infty , \gamma_{{}_a} (t) ) \cup (a,b) \cup ( \gamma_{{}_b} (t) , \infty ),
\\
\frac{x-a}{t},\ &x \in ( \gamma_{{}_a} (t) , a ),
\\
\frac{x-b}{t},\ &x \in ( b , \gamma_{{}_b} (t) ),
\end{cases}
\\
\rho &= \rho_c\ \delta_{x = \gamma_{{}_{c}} (t)} + \rho_d\ \delta_{x = \gamma_{{}_d} (t)}.
\end{aligned}
\]	

\textbf{Case 2.} $u_a > 0$, $u_b > 0$
\\	
Consider the curves defined by
	
\[
\begin{aligned}
\gamma_{{}_{a}} (t) &:= \begin{cases}
a + \sqrt{2 u_a t},\ &0 \leq t \leq t_{{}_{a,1}} := \frac{(b-a)^2}{2 u_a},
\\
\frac{a+b}{2} + \frac{u_a}{b-a} t,\ &t_{{}_{a,1}} \leq t \leq t_{{}_{a,2}} := \frac{(d-a)^2}{2 \left( \sqrt{u_a + u_b} - \sqrt{u_b} \right)^2},
\\
a + \sqrt{2 ( u_a + u_b ) t},\ &t \geq t_{{}_{a,2}},
\end{cases}
\\
\\
\gamma_{{}_{b,1}} (t) &:= \begin{cases}
b,\ &0 \leq t \leq t_{{}_{b,1}} := \frac{(b-a)^2}{2 u_a},
\\
\frac{a+b}{2} + \frac{u_a}{b-a} t,\ &t_{{}_{b,1}} \leq t \leq t_{{}_{b,2}} := \frac{(b-a)^2}{2 \left( \sqrt{u_a + u_b} - \sqrt{u_b} \right)^2},
\\
a + \sqrt{2 ( u_a + u_b ) t},\ &t \geq t_{{}_{b,2}},
\end{cases}
\\
\\
\gamma_{{}_{b,2}} (t) &:= \begin{cases}
b + \sqrt{2 u_b t},\ &0 \leq t \leq t_{{}_{b,2}},
\\
a + \sqrt{2 ( u_a + u_b ) t},\ &t \geq t_{{}_{b,2}},
\end{cases}
\\
\\
\gamma_{{}_{c}} (t) &:= \begin{cases}
c,\ &0 \leq t \leq t_{{}_{c,1}} := \frac{(c-a)^2}{2 u_a},
\\
a + \sqrt{2 u_a t},\ &t_{{}_{c,1}} \leq t \leq t_{{}_{c,2}} := \frac{(b-a)^2}{2 u_a},
\\
\frac{a+b}{2} + \frac{u_a}{b-a} t,\ &t_{{}_{c,2}} \leq t \leq t_{{}_{c,3}} := \frac{(d-a)^2}{2 \left( \sqrt{u_a + u_b} - \sqrt{u_b} \right)^2},
\\
a + \sqrt{2 ( u_a + u_b ) t},\ &t \geq t_{{}_{c,3}}.
\end{cases} 
\end{aligned}
\]
Let us also consider the point

\[
\left( x^* , t^* \right) := \left( b + \frac{\left( b-a \right) \sqrt{u_b}}{\sqrt{u_a + u_b} - \sqrt{u_b}} , \frac{(b-a)^2}{2 \left( \sqrt{u_a + u_b} - \sqrt{u_b} \right)^2} \right).
\]
Depending on the three subcases $d < x^*$, $d = x^*$ and $d > x^*$, we define a curve $x = \gamma_{{}_{d}} (t)$ in the upper-half plane as follows:

\subsection*{Subcase 1. $d < x^*$}

\[
\begin{aligned}
\gamma_{{}_{d}} (t) := \begin{cases}
d,\ &0 \leq t \leq t_{{}_{d,1}} := \frac{(d-b)^2}{2 u_b},
\\
b + \sqrt{2 u_b t},\ &t_{{}_{d,1}} \leq t \leq t_{{}_{d,2}} := t^{*},
\\
a + \sqrt{2 ( u_a + u_b ) t},\ &t \geq t_{{}_{d,2}}.
\end{cases}
\end{aligned} 
\]

\subsection*{Subcase 2. $d = x^*$}

\[
\begin{aligned}
\gamma_{{}_{d}} (t) := \begin{cases}
d,\ &0 \leq t \leq t^*,
\\
a + \sqrt{2 ( u_a + u_b ) t},\ &t \geq t^*,
\end{cases}
\end{aligned} 
\]

\subsection*{Subcase 3. $d > x^*$}

\[
\begin{aligned}
\gamma_{{}_{d}} (t) := \begin{cases}
d,\ &0 \leq t \leq t_{{}_{d}} := \frac{(d-b)^2}{2 u_b},
\\
a + \sqrt{2 ( u_a + u_b ) t},\ &t \geq t_{{}_{d}},
\end{cases}
\end{aligned}
\]
Then the velocity component $u$ and the density component $\rho$ can be explicitly described as follows:

\[
\begin{aligned}
u(x,t) &= \begin{cases}
0,\ &x \in \left( -\infty , a \right) \cup \left( \gamma_{{}_{a}} (t) , \gamma_{{}_{b,1}} (t) \right) \cup \left( \gamma_{{}_{b,2}} (t) , \infty \right),
\\
\frac{x-a}{t},\ &x \in \left( a , \gamma_{{}_{a}} (t) \right),
\\
\frac{x-b}{t},\ &x \in \left( \gamma_{{}_{b,1}} (t) , \gamma_{{}_{b,2}} (t) \right),
\end{cases}
\\
\rho &= \rho_c\ \delta_{x = \gamma_{{}_{c}} (t)} + \rho_d\ \delta_{x = \gamma_{{}_d} (t)}.
\end{aligned}
\]	

\textbf{Case 3.} $u_a > 0$, $u_b < 0$, $u_a + u_b > 0$ 
	
\begin{footnotesize}
\[
\begin{aligned}
\gamma_{{}_{a}} (t) &:= \begin{cases}
a + \sqrt{2 u_a t},\ &0 \leq t \leq t_{{}_{a,1}} := \frac{(b-a)^2}{2 \left( \sqrt{u_a} + \sqrt{- u_b} \right)^2},
\\
\frac{a+b}{2} + \frac{u_a + u_b}{b-a} t,\ &t_{{}_{a,1}} \leq t \leq t_{{}_{a,2}} := \frac{(b-a)^2}{2 ( u_a + u_b )},
\\
a + \sqrt{2 ( u_a + u_b ) t},\ &t \geq t_{{}_{a,2}},
\end{cases}
\\
\\
\gamma_{{}_{b,1}} (t) &:= \begin{cases}
b - \sqrt{- 2 u_b t},\ &0 \leq t \leq t_{{}_{b,1}} := \frac{(b-a)^2}{2 \left( \sqrt{u_a} + \sqrt{- u_b} \right)^2},
\\
\frac{a+b}{2} + \frac{u_a + u_b}{b-a} t,\ &t_{{}_{b,1}} \leq t \leq t_{{}_{b,2}} := \frac{(b-a)^2}{2 ( u_a + u_b )},
\\
a + \sqrt{2 ( u_a + u_b ) t},\ &t \geq t_{{}_{b,2}},
\end{cases}
\\
\\
\gamma_{{}_{b,2}} (t) &:= \begin{cases}
b,\ &0 \leq t \leq t_{{}_{b,2}},
\\
a + \sqrt{2 ( u_a + u_b ) t},\ &t \geq t_{{}_{b,2}}
\end{cases}
\\
\gamma_{{}_{d}} (t) &:= \begin{cases}
d,\ &0 \leq t \leq t_{{}_{d,1}} := \frac{(d-a)^2}{2 ( u_a + u_b )},
\\
a + \sqrt{2 ( u_a + u_b ) t},\ &t \geq t_{{}_{d,1}}.
\end{cases}
\end{aligned}
\]
\end{footnotesize}
To proceed further, we have to consider the point

\[
\begin{aligned}
\left( x^* , t^* \right) := \left( a + \frac{\left( b-a \right) \sqrt{u_a}}{\sqrt{u_a} + \sqrt{- u_b}} , \frac{(b-a)^2}{2 \left( \sqrt{u_a} + \sqrt{- u_b} \right)^2} \right),
\end{aligned}
\]
and depending on the three subcases $c < x^*$, $c = x^*$ and $c > x^*$, we define a curve $x = \gamma_{{}_{c}} (t)$ in the upper-half plane as follows:

\subsection*{Subcase 1. $c < x^*$}

\[
\begin{aligned}
\gamma_{{}_{c}} (t) := \begin{cases}
c,\ &0 \leq t \leq t_{{}_{c,1}} := \frac{(c-a)^2}{2 u_a},
\\
a + \sqrt{2 u_a t},\ &t_{{}_{c,1}} \leq t \leq t_{{}_{c,2}} := t^*,
\\
\frac{a+b}{2} + \frac{u_a + u_b}{b-a} t,\ &t_{{}_{c,2}} \leq t \leq t_{{}_{c,3}} := \frac{(b-a)^2}{2 ( u_a + u_b )},
\\
a + \sqrt{2 ( u_a + u_b ) t},\ &t \geq t_{{}_{c,3}}.
\end{cases}
\end{aligned}
\]

\subsection*{Subcase 2. $c = x^*$}

\[
\begin{aligned}
\gamma_{{}_{c}} (t) := \begin{cases}
c,\ &0 \leq t \leq t_{{}_{c,1}} := t^*,
\\
\frac{a+b}{2} + \frac{u_a + u_b}{b-a} t,\ &t_{{}_{c,1}} \leq t \leq t_{{}_{c,2}} := \frac{(b-a)^2}{2 ( u_a + u_b )},
\\
a + \sqrt{2 ( u_a + u_b ) t},\ &t \geq t_{{}_{c,2}}.
\end{cases}
\end{aligned}
\]

\subsection*{Subcase 3. $c > x^*$}

\[
\begin{aligned}
\gamma_{{}_{c}} (t) := \begin{cases}
c,\ &0 \leq t \leq t_{{}_{c,1}} := \frac{(b-c)^2}{- 2 u_b},
\\
b - \sqrt{- 2 u_b t},\ &t_{{}_{c,1}} \leq t \leq t_{{}_{c,2}} := t^*,
\\
\frac{a+b}{2} + \frac{u_a + u_b}{b-a} t,\ &t_{{}_{c,2}} \leq t \leq t_{{}_{c,3}} := \frac{(b-a)^2}{2 ( u_a + u_b )},
\\
a + \sqrt{2 ( u_a + u_b ) t},\ &t \geq t_{{}_{c,3}}.
\end{cases}
\end{aligned}
\]
Then the velocity component $u$ and the density component $\rho$ can be explicitly described as follows:

\[
\begin{aligned}
u(x,t) &= \begin{cases}
0,\ &x \in \left( -\infty , a \right) \cup \left( \gamma_{{}_{a}} (t) , \gamma_{{}_{b,1}} (t) \right) \cup \left( \gamma_{{}_{b,2}} (t) , \infty \right),
\\
\frac{x-a}{t},\ &x \in \left( a , \gamma_{{}_{a}} (t) \right) ,
\\
\frac{x-b}{t},\ &x \in \left( \gamma_{{}_{b,1}} (t) , \gamma_{{}_{b,2}} (t) \right),
\end{cases}
\\
\rho &= \rho_c\ \delta_{x = \gamma_{{}_{c}} (t)} + \rho_d\ \delta_{x = \gamma_{{}_d} (t)}.
\end{aligned}
\]

\textbf{Case 4.} $u_a > 0$, $u_b < 0$, $u_a + u_b < 0$
\\	
Consider the curves defined by
	
\[
\begin{aligned}
\gamma_{{}_{a,1}} (t) &:= \begin{cases}
a,\ &0 \leq t \leq t_{{}_{a,2}} := \frac{(b-a)^2}{- 2 ( u_a + u_b )},
\\
b - \sqrt{- 2 ( u_a + u_b ) t},\ &t \geq t_{{}_{a,2}},
\end{cases}
\\
\\
\gamma_{{}_{a,2}} (t) &:= \begin{cases}
a + \sqrt{2 u_a t},\ &0 \leq t \leq t_{{}_{a,1}} := \frac{(b-a)^2}{2 \left( \sqrt{u_a} + \sqrt{- u_b} \right)^2},
\\
\frac{a+b}{2} + \frac{u_a + u_b}{b-a} t,\ &t_{{}_{a,1}} \leq t \leq t_{{}_{a,2}},
\\
b - \sqrt{- 2 ( u_a + u_b ) t},\ &t \geq t_{{}_{a,2}},
\end{cases}
\\
\\
\gamma_{{}_{b}} (t) &:= \begin{cases}
b - \sqrt{- 2 u_b t},\ &0 \leq t \leq t_{{}_{b,1}} := \frac{(b-a)^2}{2 \left( \sqrt{u_a} + \sqrt{- u_b} \right)^2}
\\
\frac{a+b}{2} + \frac{u_a + u_b}{b-a} t,\ &t_{{}_{b,1}} \leq t \leq t_{{}_{b,2}} := \frac{(b-a)^2}{- 2 ( u_a + u_b )},
\\
b - \sqrt{- 2 ( u_a + u_b ) t},\ &t \geq t_{{}_{b,2}},
\end{cases}
\\
\\
\gamma_{{}_{d}} (t) &:= d,\ t \geq 0.
\end{aligned}
\]
Let us also consider the coordinates

\[
\begin{aligned}
\left( x^* , t^* \right) := \left( a + \frac{\left( b-a \right) \sqrt{u_a}}{\sqrt{u_a} + \sqrt{- u_b}} , \frac{(b-a)^2}{2 \left( \sqrt{u_a} + \sqrt{-u_b} \right)^2} \right),
\end{aligned}
\]
and depending on three subcases $c < x^*$, $c = x^*$ and $c > x^*$, we introduce a curve $x = \gamma_{{}_{c}} (t)$ in the upper-half plane as follows:

\subsection*{Subcase 1. $c < x^*$}

\[
\begin{aligned}
\gamma_{{}_{c}} (t) := \begin{cases}
c,\ &0 \leq t \leq t_{{}_{c,1}} := \frac{(c-a)^2}{2 u_a},
\\
a + \sqrt{2 u_a t},\ &t_{{}_{c,1}} \leq t \leq t_{{}_{c,2}} := t^*,
\\
\frac{a+b}{2} + \frac{u_a + u_b}{b-a} t,\ &t_{{}_{c,2}} \leq t \leq t_{{}_{c,3}} := \frac{(b-a)^2}{- 2 ( u_a + u_b )},
\\
b - \sqrt{- 2 ( u_a + u_b ) t},\ &t \geq t_{{}_{c,3}}.
\end{cases}
\end{aligned}
\]

\subsection*{Subcase 2. $c = x^*$}

\[
\begin{aligned}
\gamma_{{}_{c}} (t) := \begin{cases}
c,\ &0 \leq t \leq t_{{}_{c,1}} := t^*,
\\
\frac{a+b}{2} + \frac{u_a + u_b}{b-a} t,\ &t_{{}_{c,1}} \leq t \leq t_{{}_{c,2}} := \frac{(b-a)^2}{- 2 ( u_a + u_b )},
\\
b - \sqrt{- 2 ( u_a + u_b ) t},\ &t \geq t_{{}_{c,2}}.
\end{cases}
\end{aligned}
\]

\subsection*{Subcase 3. $c > x^*$}

\[
\begin{aligned}
\gamma_{{}_{c}} (t) := \begin{cases}
c,\ &0 \leq t \leq t_{{}_{c,1}} := \frac{(b-c)^2}{- 2 u_b},
\\
b - \sqrt{- 2 u_b t},\ &t_{{}_{c,1}} \leq t \leq t_{{}_{c,2}} := t^*,
\\
\frac{a+b}{2} + \frac{u_a + u_b}{b-a} t,\ &t_{{}_{c,2}} \leq t \leq t_{{}_{c,3}} := \frac{(b-a)^2}{- 2 ( u_a + u_b )},
\\
b - \sqrt{- 2 ( u_a + u_b ) t},\ &t \geq t_{{}_{c,3}}.
\end{cases}
\end{aligned}
\]
Then the velocity component $u$ and the density component $\rho$ can be explicitly described as follows:

\[
\begin{aligned}
u(x,t) &= \begin{cases}
0,\ &x \in \left( -\infty , \gamma_{{}_{a,1}} (t) \right) \cup \left( \gamma_{{}_{a,2}} (t) , \gamma_{{}_{b}} (t) \right) \cup \left( b , \infty \right),
\\
\frac{x-a}{t},\ &x \in \left( \gamma_{{}_{a,1}} (t) , \gamma_{{}_{a,2}} (t) \right),
\\
\frac{x-b}{t},\ &x \in \left( \gamma_{{}_{b}} (t) , b \right),
\end{cases}
\\
\rho &= \rho_c\ \delta_{x = \gamma_{{}_{c}} (t)} + \rho_d\ \delta_{x = \gamma_{{}_d} (t)}.
\end{aligned}
\]

\textbf{Case 5.} $u_a > 0$, $u_b < 0$, $u_a + u_b = 0$
\\	
Consider the curves defined by
	
\[
\begin{aligned}
\gamma_{{}_{a}} (t) &:= \begin{cases} 
a + \sqrt{2 u_a t},\ &0 \leq t \leq t_{{}_{a,1}} := \frac{(b-a)^2}{2 \left( \sqrt{u_a} + \sqrt{- u_b} \right)^2},
\\
\frac{a+b}{2},\ &t \geq t_{{}_{a,1}},
\end{cases}
\\
\\
\gamma_{{}_{b}} (t) &:= \begin{cases} 
b - \sqrt{- 2 u_b t},\ &0 \leq t \leq t_{{}_{b,1}} := \frac{(b-a)^2}{2 \left( \sqrt{u_a} + \sqrt{- u_b} \right)^2},
\\
\frac{a+b}{2},\ &t \geq t_{{}_{b,1}}.
\end{cases} 
\end{aligned}
\]

Depending on the three subcases $c < \frac{a+b}{2}$, $c = \frac{a+b}{2}$ and $c > \frac{a+b}{2}$, we consider a curve $x = \gamma_{{}_{c}} (t)$ defined in the upper-half plane as follows:

\subsection*{Subcase 1. $c < \frac{a+b}{2}$}
 
\[
\begin{aligned}
\gamma_{{}_{c}} (t) := \begin{cases}
c,\ &0 \leq t \leq t_{{}_{c,1}} := \frac{(c-a)^2}{2 u_a},
\\
a + \sqrt{2 u_a t},\ &t_{{}_{c,1}} \leq t \leq t_{{}_{c,2}} := \frac{(b-a)^2}{2 \left( \sqrt{u_a} + \sqrt{- u_b} \right)^2},
\\
\frac{a+b}{2},\ &t \geq t_{{}_{c,2}}.
\end{cases}
\end{aligned}
\]

\subsection*{Subcase 2. $c = \frac{a+b}{2}$}
\[
\gamma_{{}_{c}} (t):= \frac{a+b}{2},\ t \geq 0.
\]

\subsection*{Subcase 3. $c > \frac{a+b}{2}$}

\[
\begin{aligned}
\gamma_{{}_{c}} (t) := \begin{cases}
c,\ &0 \leq t \leq t_{{}_{c,1}} := \frac{(b-c)^2}{- 2 u_b},
\\
b - \sqrt{- 2 u_b t},\ &t_{{}_{c,1}} \leq t \leq t_{{}_{c,2}} := \frac{(b-a)^2}{2 \left( \sqrt{u_a} + \sqrt{- u_b} \right)^2},
\\
\frac{a+b}{2},\ &t \geq t_{{}_{c,2}}.
\end{cases}
\end{aligned}
\]
Then the velocity component $u$ and the density component $\rho$ have the following explicit representations:

\[
\begin{aligned}
u(x,t) &= \begin{cases}
0,\ &x \in \left( -\infty , a \right) \cup \left( \gamma_{{}_{a}} (t) , \gamma_{{}_{b}} (t) \right) \cup \left( b , \infty \right),
\\
\frac{x-a}{t},\ &x \in \left( a , \gamma_{{}_{a}} (t) \right),
\\
\frac{x-b}{t},\ &x \in \left( \gamma_{{}_{b}} (t) , b \right),
\end{cases}
\\
\rho &= \rho_c\ \delta_{x = \gamma_{{}_{c}} (t)} + \rho_d\ \delta_{x = \gamma_{{}_d} (t)}.
\end{aligned}
\]

\textbf{Case 6.} $u_a < 0$, $u_b < 0$
\\	
Consider the curves defined by
	
\[
\begin{aligned}
\gamma_{{}_{a,1}} (t)
	&:= \begin{cases}
a - \sqrt{- 2 u_a t},\ &0 \leq t \leq t_{a,1} := \frac{(b-a)^2}{2 \left( \sqrt{- ( u_a + u_b )} - \sqrt{- u_a} \right)^2},
	\\
b - \sqrt{- 2 ( u_a + u_b ) t},\ &t \geq t_{a,1},
\end{cases}	
	\\
	\\
\gamma_{{}_{a,2}} (t)
	&:= \begin{cases}
a,\ &0 \leq t \leq t_{a,2} := \frac{(b-a)^2}{- 2 u_b},
	\\
\frac{a+b}{2} + \frac{u_b}{b-a} t,\ &t_{a,2} \leq t \leq t_{a,3} := \frac{(b-a)^2}{2 \left( \sqrt{- ( u_a + u_b )} - \sqrt{- u_a} \right)^2},
	\\
b - \sqrt{- 2 ( u_a + u_b ) t},\ &t \geq t_{a,3},
\end{cases}	
	\\
	\\
\gamma_{{}_b} (t)
	&:= \begin{cases}
b - \sqrt{- 2 u_b t},\ &0 \leq t \leq t_{b,1} := \frac{(b-a)^2}{- 2 u_b},
	\\
\frac{a+b}{2} + \frac{u_b}{b-a} t,\ &t_{b,1} \leq t \leq t_{b,2} := \frac{(b-a)^2}{2 \left( \sqrt{- ( u_a + u_b )} - \sqrt{- u_a} \right)^2},
	\\
b - \sqrt{- 2 ( u_a + u_b ) t},\ &t \geq t_{b,2},
\end{cases}	
	\\
	\\
\gamma_{{}_c} (t)
	&:= \begin{cases}
c,\ &0 \leq t \leq t_{c,1} := \frac{(b-c)^2}{- 2 u_b},	
	\\
b - \sqrt{- 2 u_b t},\ &t_{c,1} \leq t \leq t_{c,2} := \frac{(b-a)^2}{- 2 u_b},
	\\
\frac{a+b}{2} + \frac{u_b}{b-a} t,\ &t_{c,2} \leq t \leq t_{c,3} := \frac{(b-a)^2}{2 \left( \sqrt{- ( u_a + u_b )} - \sqrt{- u_a} \right)^2},
	\\
b - \sqrt{- 2 ( u_a + u_b ) t},\ &t \geq t_{c,3},
\end{cases}	
\\
\\
\gamma_{{}_{d}} (t) &:= d,\ t \geq 0. 
\end{aligned}
\]
Then the velocity component $u$ and the density component $\rho$ can be explicitly described as follows:

\[
\begin{aligned}
u(x,t) &= \begin{cases}
0,\ &x \in \Big( -\infty , \gamma_{{}_{a,1}} (t) \Big) \cup \Big( \gamma_{{}_{a,2}} (t) , \gamma_{{}_b} (t) \Big) \cup \Big( b , \infty \Big),
\\
\frac{x-a}{t},\ &x \in \Big( \gamma_{{}_{a,1}} (t) , \gamma_{{}_{a,2}} (t) \Big),
\\
\frac{x-b}{t},\ & x \in \Big( \gamma_{{}_b} (t) , b \Big),
\end{cases}
\\
\rho &= \rho_c\ \delta_{x = \gamma_{{}_{c}} (t)} + \rho_d\ \delta_{x = \gamma_{{}_d} (t)}.
\end{aligned}
\]

\end{theorem}

\begin{proof}

The first step is to consider the modified adhesion model corresponding to the problem \eqref{comms-1}-\eqref{comms-2}, namely

\begin{equation}
u^\epsilon_t + \left({(u^\epsilon)^2}/{2} \right)_x =\frac{\epsilon}{2} u^\epsilon_{xx},\ \rho^\epsilon_t + (\rho^\epsilon u^\epsilon)_x = \frac{\epsilon}{2} \rho^\epsilon_{xx},
\label{comms-3}
\end{equation}
with initial data

\begin{equation}
u^\epsilon|_{t=0} = u_a\ \delta_{x=a} + u_b\ \delta_{x=b},\ \rho^\epsilon|_{t=0} = \rho_c\ \delta_{x=c} +\rho_d\ \delta_{x=d}.
\label{comms-4}
\end{equation}
So, if $\left( U^\epsilon , R^\epsilon \right)$ is a solution to the system

\begin{equation}
U_t^\epsilon + \frac{\left( U_x^\epsilon \right)^2}{2} = \frac{\epsilon}{2} U_{xx}^\epsilon,\ R_t^\epsilon + R_x^\epsilon U_x^\epsilon = \frac{\epsilon}{2} R_{xx}^\epsilon
\label{comms-5}
\end{equation}
under the initial conditions

\begin{equation}
\begin{aligned}
\\
U^\epsilon (x,0) &= \begin{cases}
0,\ &{x < a},
\\
u_a,\ &{a < x < b},
\\
u_a + u_b,\ &{x > b},
\end{cases}
\\
R^\epsilon (x,0) &= \begin{cases}
0,\ &{x < c},
\\
\rho_c,\ &{c < x < d},
\\
\rho_c + \rho_d,\ &{x > d},
\end{cases}
\end{aligned}
\label{comms-6}
\end{equation}
then $\left( \overline{u}^\epsilon , \overline{\rho}^\epsilon \right) = \left( U_x^\epsilon , R_x^\epsilon \right)$ will solve the problem \eqref{comms-3}-\eqref{comms-4}.

Using the generalized Hopf-Cole transformations

\[
V^\epsilon := e^{- \frac{U^\epsilon}{\epsilon}},\ S^\epsilon := R^\epsilon\ e^{- \frac{U^\epsilon}{\epsilon}},
\]
we see that $V^\epsilon$ and $S^\epsilon$ satisfy

\begin{equation}
V_t^\epsilon = \frac{\epsilon}{2} V_{xx}^\epsilon,\ S_t^\epsilon = \frac{\epsilon}{2} S_{xx}^\epsilon
\label{comms-7}
\end{equation}
under the initial conditions

\begin{align}
V^\epsilon (x,0) &= \begin{cases}
1,\ &{ x < a },
\\
e^{-\frac{u_a}{\epsilon}},\ &{ a < x < b },
\\
e^{-\frac{u_a + u_b}{\epsilon}},\ &{ x > b },
\end{cases}
\label{comms-8}
\\
S^\epsilon (x,0) &= \begin{cases}
0,\ &{ x < c },
\\
\rho_c\ e^{-\frac{u_a}{\epsilon}},\ &{ c < x < b },
\\
\rho_c\ e^{-\frac{u_a + u_b}{\epsilon}},\ &{ b < x < d },
\\
( \rho_c + \rho_d )\ e^{-\frac{u_a + u_b}{\epsilon}},\ &{ x > d }.
\end{cases}
\label{comms-9}
\end{align}
Therefore, the explicit solutions to the linear system \eqref{comms-7} under the initial conditions \eqref{comms-8}-\eqref{comms-9} are given by

\[
\begin{aligned}
\begin{pmatrix}
V^\epsilon (x,t)
\\
\\
S^\epsilon (x,t)
\end{pmatrix} = \begin{pmatrix}
\frac{1}{\sqrt{\pi}} \left[
\sqrt{\pi} - \textnormal{erfc} \left( - \frac{x-a}{\sqrt{2t \epsilon}} \right) + e^{- \frac{u_a}{\epsilon}} \left( \textnormal{erfc} \left( - \frac{x-a}{\sqrt{2t \epsilon}} \right) - \textnormal{erfc} \left( - \frac{x-b}{\sqrt{2t \epsilon}} \right) \right) \right.
\\
+ \left. e^{- \frac{u_a + u_b}{\epsilon}} \textnormal{erfc} \left( - \frac{x-b}{\sqrt{2t \epsilon}} \right) \right]
\\
\\
\frac{1}{\sqrt{\pi}} \left[ \rho_c\ e^{- \frac{u_a}{\epsilon}}\ \left( \textnormal{erfc} \left( - \frac{x-c}{\sqrt{2t \epsilon}} \right) - \textnormal{erfc} \left( - \frac{x-b}{\sqrt{2t \epsilon}} \right) \right) \right.
\\
+ \rho_c\ e^{- \frac{u_a + u_b}{\epsilon}}\ \left( \textnormal{erfc} \left( - \frac{x-b}{\sqrt{2t \epsilon}} \right) - \textnormal{erfc} \left( - \frac{x-d}{\sqrt{2t \epsilon}} \right) \right)
\\
+ \left. ( \rho_c + \rho_d )\ e^{- \frac{u_a + u_b}{\epsilon}}\ \textnormal{erfc} \left( - \frac{x-d}{\sqrt{2t \epsilon}} \right) \right]
\end{pmatrix},
\end{aligned}
\]
where

\[
\textnormal{erfc} \left( z \right) := \int_{z}^{\infty}\ e^{-t^2}\ dt,\ z \in \textbf{R}^1.
\]
Coming back to the problems \eqref{comms-3}-\eqref{comms-4} and \eqref{comms-5}-\eqref{comms-6}, it follows that we have the following explicit representations for $u^\epsilon$ and $R^\epsilon$:

\begin{equation}
\begin{aligned}
u^\epsilon (x,t) 
		&= \frac{\epsilon}{\sqrt{2t \epsilon}} \cdot \frac{\begin{aligned}
e^{ - \frac{(x-a)^2}{2t \epsilon}} \left( 1 - e^{- \frac{u_a}{\epsilon}} \right) + e^{-\frac{(x-b)^2}{2t \epsilon}} \left( e^{- \frac{u_a}{\epsilon}} - e^{- \frac{u_a + u_b}{\epsilon}} \right)
\end{aligned}}{
\begin{aligned}
\sqrt{\pi} - \textnormal{erfc} \left( - \frac{x-a}{\sqrt{2t \epsilon}} \right) + e^{- \frac{u_a}{\epsilon}}\ \left( \textnormal{erfc} \left( - \frac{x-a}{\sqrt{2t \epsilon}} \right) - \textnormal{erfc} \left( - \frac{x-b}{\sqrt{2t \epsilon}} \right) \right) 
\\
+ e^{- \frac{u_a + u_b}{\epsilon}}\ \textnormal{erfc} \left( - \frac{x-b}{\sqrt{2t \epsilon}} \right) 
\end{aligned} },
\label{intro-11}
\end{aligned}
\end{equation}
\\
\\
\begin{equation}
\begin{aligned}
R^\epsilon (x,t) &= \frac{
\begin{aligned}
\rho_c\ e^{- \frac{u_a}{\epsilon}}\ \left( \textnormal{erfc} \left( - \frac{x-c}{\sqrt{2t \epsilon}} \right) - \textnormal{erfc} \left( - \frac{x-b}{\sqrt{2t \epsilon}} \right) \right) 
\\
+ \rho_c\ e^{- \frac{u_a + u_b}{\epsilon}}\ \left( \textnormal{erfc} \left( - \frac{x-b}{\sqrt{2t \epsilon}} \right) - \textnormal{erfc} \left( - \frac{x-d}{\sqrt{2t \epsilon}} \right) \right)
\\
+ ( \rho_c + \rho_d )\ e^{- \frac{u_a + u_b}{\epsilon}}\ \textnormal{erfc} \left( - \frac{x-d}{\sqrt{2t \epsilon}} \right)
\end{aligned} 
}{ 
\begin{aligned}
\sqrt{\pi} - \textnormal{erfc} \left( - \frac{x-a}{\sqrt{2t \epsilon}} \right) + e^{- \frac{u_a}{\epsilon}}\ \left( \textnormal{erfc} \left( - \frac{x-a}{\sqrt{2t \epsilon}} \right) - \textnormal{erfc} \left( - \frac{x-b}{\sqrt{2t \epsilon}} \right) \right) 
\\
+ e^{- \frac{u_a + u_b}{\epsilon}}\ \textnormal{erfc} \left( - \frac{x-b}{\sqrt{2t \epsilon}} \right) 
\end{aligned} }.
\end{aligned}
\label{intro-12}
\end{equation}
In order to discuss the passage to the limit as $\epsilon \rightarrow 0$, we will use the following properties of the function erfc:

\begin{enumerate}
	
	\item ${\displaystyle{\lim_{z \rightarrow \infty}}} \textnormal{erfc} \left( z \right) = 0$.
	
	\item $\textnormal{erfc} \left( z \right) = \left( \frac{1}{2 z} - \frac{1}{4 z^3} + o \left( \frac{1}{z^3} \right) \right)\ e^{- z^2} \textnormal{ as } z \rightarrow \infty$.
	
	\item ${\displaystyle{\lim_{z \rightarrow \infty}}}\ f(z) = \frac{1}{2}$,

\end{enumerate}
where $f : z \mapsto z\ \textnormal{erfc} \left( z \right)\ e^{z^2}$ for every $z \in \textbf{R}^1$. These properties have been proved in the appendix.

Here we discuss another important property of erfc that will be useful later on. For distinct $p,q,x \in \textbf{R}^1$, consider $P_\epsilon = P_\epsilon (x,t) := \frac{|x-p|}{\sqrt{2t \epsilon}}$ and $Q_\epsilon = Q_\epsilon (x,t) := \frac{|x-q|}{\sqrt{2t \epsilon}}$ for each $\epsilon > 0$, $t>0$. Then using property (3) above, we have
\[
\begin{aligned}
\lim_{\epsilon \rightarrow 0} \frac{\textnormal{erfc} \left( P_\epsilon \right)}{\textnormal{erfc} \left( Q_\epsilon \right)} 
		&= \lim_{\epsilon \rightarrow 0} \frac{f(P_\epsilon)}{f(Q_\epsilon)} \cdot \frac{Q_\epsilon\ e^{Q_\epsilon^2}}{P_\epsilon\ e^{P_\epsilon^2}} 
		\\
		&= \lim_{\epsilon \rightarrow 0} \frac{|x-q|\ f(P_\epsilon)}{|x-p|\ f(Q_\epsilon)} \cdot e^{Q_\epsilon^2 - P_\epsilon^2}
		\\
		&= \begin{cases}
		0,\ &|x-p| > |x-q|,
		\\
		\infty,\ &|x-p| < |x-q|.
		\end{cases}
\end{aligned}
\]
In the computations that follow, we use the following abbreviations for our functions and variables under study:

\[
\begin{aligned}
A_\epsilon = A_\epsilon (x,t) = \frac{|x-a|}{\sqrt{2t \epsilon}},\ B_\epsilon = B_\epsilon (x,t) = \frac{|x-b|}{\sqrt{2t \epsilon}}, 
\\
C_\epsilon = C_\epsilon (x,t) = \frac{|x-c|}{\sqrt{2t \epsilon}},\ D_\epsilon = D_\epsilon (x,t) = \frac{|x-d|}{\sqrt{2t \epsilon}},
\\
u^\epsilon = u^\epsilon (x,t),\ R^\epsilon = R^\epsilon (x,t).
\end{aligned}
\]
First we look into the explicit expressions of $u^\epsilon$ and $R^\epsilon$ in different regions in terms of the new variables $A_\epsilon$, $B_\epsilon$, $C_\epsilon$ and $D_\epsilon$:

\begin{footnotesize}
\subsection*{(1) $x<a$}

\[
\begin{aligned}
	u^\epsilon &= 
		\frac{\epsilon}{\sqrt{2t \epsilon}} \cdot \frac{\begin{aligned}
e^{-{A_\epsilon}^2} \left( 1 - e^{- \frac{u_a}{\epsilon}} \right) + e^{-{B_\epsilon}^2} \left( e^{- \frac{u_a}{\epsilon}} - e^{- \frac{u_a + u_b}{\epsilon}} \right)
\end{aligned}}{
\begin{aligned}
\sqrt{\pi} + {\textnormal{erfc}} \left( {A_\epsilon} \right) \left( e^{- \frac{u_a}{\epsilon}} - 1 \right) + {\textnormal{erfc}} \left( {B_\epsilon} \right) \left( e^{- \frac{u_a + u_b}{\epsilon}} - e^{- \frac{u_a}{\epsilon}} \right)
\end{aligned} },
\\
\\
	R^\epsilon &= 
		\frac{
\begin{aligned}
	\rho_c\ \textnormal{erfc} \left( B_\epsilon \right) \left( e^{- \frac{u_a + u_b}{\epsilon}} - e^{- \frac{u_a}{\epsilon}} \right) 
	\\
	+ \rho_c\ \textnormal{erfc} \left( C_\epsilon \right) e^{- \frac{u_a}{\epsilon}} + \rho_d\ \textnormal{erfc} \left( D_\epsilon \right) e^{- \frac{u_a + u_b}{\epsilon}}
\end{aligned} 
}{ 
\begin{aligned}
\sqrt{\pi} + {\textnormal{erfc}} \left( {A_\epsilon} \right) \left( e^{- \frac{u_a}{\epsilon}} - 1 \right) + {\textnormal{erfc}} \left( {B_\epsilon} \right) \left( e^{- \frac{u_a + u_b}{\epsilon}} - e^{- \frac{u_a}{\epsilon}} \right) 
\end{aligned} }.
\end{aligned}
\]

\subsection*{(2) $a < x < c$}

\[
\begin{aligned}
	u^\epsilon &= 
		\frac{\epsilon}{\sqrt{2t \epsilon}} \cdot \frac{\begin{aligned}
e^{-{A_\epsilon}^2} \left( 1 - e^{- \frac{u_a}{\epsilon}} \right) + e^{-{B_\epsilon}^2} \left( e^{- \frac{u_a}{\epsilon}} - e^{- \frac{u_a + u_b}{\epsilon}} \right)
\end{aligned}}{
\begin{aligned}
\sqrt{\pi}\ e^{- \frac{u_a}{\epsilon}} + {\textnormal{erfc}} \left( {A_\epsilon} \right) \left( 1 - e^{- \frac{u_a}{\epsilon}} \right) + {\textnormal{erfc}} \left( {B_\epsilon} \right) \left( e^{- \frac{u_a + u_b}{\epsilon}} - e^{- \frac{u_a}{\epsilon}} \right)
\end{aligned} },
\\
\\
	R^\epsilon &= 
		\frac{
\begin{aligned}
	\rho_c\ \textnormal{erfc} \left( B_\epsilon \right) \left( e^{- \frac{u_a + u_b}{\epsilon}} - e^{- \frac{u_a}{\epsilon}} \right) 
	\\
	+ \rho_c\ \textnormal{erfc} \left( C_\epsilon \right) e^{- \frac{u_a}{\epsilon}} + \rho_d\ \textnormal{erfc} \left( D_\epsilon \right) e^{- \frac{u_a + u_b}{\epsilon}}
\end{aligned} 
}{ 
\begin{aligned}
\sqrt{\pi}\ e^{- \frac{u_a}{\epsilon}} + {\textnormal{erfc}} \left( {A_\epsilon} \right) \left( 1 - e^{- \frac{u_a}{\epsilon}} \right) + {\textnormal{erfc}} \left( {B_\epsilon} \right) \left( e^{- \frac{u_a + u_b}{\epsilon}} - e^{- \frac{u_a}{\epsilon}} \right)
\end{aligned} }.
\end{aligned}
\]

\subsection*{(3) $c < x < b$}

\[
\begin{aligned}
	u^\epsilon &= 
		\frac{\epsilon}{\sqrt{2t 	\epsilon}} \cdot \frac{\begin{aligned}
e^{-{A_\epsilon}^2} \left( 1 - e^{- \frac{u_a}{\epsilon}} \right) + e^{-{B_\epsilon}^2} \left( e^{- \frac{u_a}{\epsilon}} - e^{- \frac{u_a + u_b}{\epsilon}} \right)
\end{aligned}}{
\begin{aligned}
\sqrt{\pi}\ e^{- \frac{u_a}{\epsilon}} + {\textnormal{erfc}} \left( {A_\epsilon} \right) \left( 1 - e^{- \frac{u_a}{\epsilon}} \right) + {\textnormal{erfc}} \left( {B_\epsilon} \right) \left( e^{- \frac{u_a + u_b}{\epsilon}} - e^{- \frac{u_a}{\epsilon}} \right)
\end{aligned} },
\\
\\
	R^\epsilon &= 
		\frac{
\begin{aligned}
	\sqrt{\pi}\ \rho_c\ e^{- \frac{u_a}{\epsilon}} + \rho_c\ \textnormal{erfc} \left( B_\epsilon \right) \left( e^{- \frac{u_a + u_b}{\epsilon}} - e^{- \frac{u_a}{\epsilon}} \right) 
	\\
	- \rho_c\ \textnormal{erfc} \left( C_\epsilon \right) e^{- \frac{u_a}{\epsilon}} + \rho_d\ \textnormal{erfc} \left( D_\epsilon \right) e^{- \frac{u_a + u_b}{\epsilon}}
\end{aligned} 
}{ 
\begin{aligned}
\sqrt{\pi}\ e^{- \frac{u_a}{\epsilon}} + {\textnormal{erfc}} \left( {A_\epsilon} \right) \left( 1 - e^{- \frac{u_a}{\epsilon}} \right) + {\textnormal{erfc}} \left( {B_\epsilon} \right) \left( e^{- \frac{u_a + u_b}{\epsilon}} - e^{- \frac{u_a}{\epsilon}} \right)
\end{aligned} }.
\end{aligned}
\]

\subsection*{(4) $b < x < d$}

\[
\begin{aligned}
	u^\epsilon &= 
		\frac{\epsilon}{\sqrt{2t \epsilon}} \cdot \frac{\begin{aligned}
e^{-{A_\epsilon}^2} \left( 1 - e^{- \frac{u_a}{\epsilon}} \right) + e^{-{B_\epsilon}^2} \left( e^{- \frac{u_a}{\epsilon}} - e^{- \frac{u_a + u_b}{\epsilon}} \right)
\end{aligned}}{
\begin{aligned}
\sqrt{\pi}\ e^{- \frac{u_a + u_b}{\epsilon}} + {\textnormal{erfc}} \left( {A_\epsilon} \right) \left( 1 - e^{- \frac{u_a}{\epsilon}} \right) + {\textnormal{erfc}} \left( {B_\epsilon} \right) \left( e^{- \frac{u_a}{\epsilon}} - e^{- \frac{u_a + u_b}{\epsilon}} \right) 
\end{aligned} },
\\
\\
	R^\epsilon &= 
		\frac{
\begin{aligned}
	\sqrt{\pi}\ \rho_c\ e^{- \frac{u_a + u_b}{\epsilon}} + \rho_c\ \textnormal{erfc} \left( B_\epsilon \right) \left( e^{- \frac{u_a}{\epsilon}} - e^{- \frac{u_a + u_b}{\epsilon}} \right) 
\\
	- \rho_c\ \textnormal{erfc} \left( C_\epsilon \right) e^{- \frac{u_a}{\epsilon}} + \rho_d\ \textnormal{erfc} \left( D_\epsilon \right) e^{- \frac{u_a + u_b}{\epsilon}} 	
\end{aligned} 
}{ 
\begin{aligned}
	\sqrt{\pi}\ e^{- \frac{u_a + u_b}{\epsilon}} + {\textnormal{erfc}} \left( {A_\epsilon} \right) \left( 1 - e^{- \frac{u_a}{\epsilon}} \right) + {\textnormal{erfc}} \left( {B_\epsilon} \right) \left( e^{- \frac{u_a}{\epsilon}} - e^{- \frac{u_a + u_b}{\epsilon}} \right) 
\end{aligned} }.
\end{aligned}
\]

\subsection*{(5) $x > d$}

\[
\begin{aligned}
	u^\epsilon &= 
		\frac{\epsilon}{\sqrt{2t \epsilon}} \cdot \frac{\begin{aligned}
e^{-{A_\epsilon}^2} \left( 1 - e^{- \frac{u_a}{\epsilon}} \right) + e^{-{B_\epsilon}^2} \left( e^{- \frac{u_a}{\epsilon}} - e^{- \frac{u_a + u_b}{\epsilon}} \right)
\end{aligned}}{
\begin{aligned}
\sqrt{\pi}\ e^{- \frac{u_a + u_b}{\epsilon}} + {\textnormal{erfc}} \left( {A_\epsilon} \right) \left( 1 - e^{- \frac{u_a}{\epsilon}} \right) + {\textnormal{erfc}}({B_\epsilon}) \left( e^{- \frac{u_a}{\epsilon}} - e^{- \frac{u_a + u_b}{\epsilon}} \right) 
\end{aligned} },
\\
\\
	R^\epsilon &= 
		\frac{
\begin{aligned}
	\sqrt{\pi}\ ( \rho_c + \rho_d )\ e^{- \frac{u_a + u_b}{\epsilon}} + \rho_c\ \textnormal{erfc} \left( B_\epsilon \right) \left( e^{- \frac{u_a}{\epsilon}} - e^{- \frac{u_a + u_b}{\epsilon}} \right) 
\\
- \rho_c\ \textnormal{erfc} \left( C_\epsilon \right) e^{- \frac{u_a}{\epsilon}} -\rho_d\ \textnormal{erfc} \left( D_\epsilon \right)\ e^{- \frac{u_a + u_b}{\epsilon}}
\end{aligned} 
}{ 
\begin{aligned}
	\sqrt{\pi}\ e^{- \frac{u_a + u_b}{\epsilon}} + {\textnormal{erfc}} \left( {A_\epsilon} \right) \left( 1 - e^{- \frac{u_a}{\epsilon}} \right) + {\textnormal{erfc}}({B_\epsilon}) \left( e^{- \frac{u_a}{\epsilon}} - e^{- \frac{u_a + u_b}{\epsilon}} \right)
\end{aligned} }.
\end{aligned}
\]
\end{footnotesize}
Let us consider various cases depending on the relative signs of $u_a$ and $u_b$.

\textbf{Case 1.} $u_a < 0$, $u_b > 0$

The limit $\lim_{\epsilon \rightarrow 0} \left( u^\epsilon , R^\epsilon \right)$ is evaluated explicitly in different regions as follows:

\begin{footnotesize}
\begin{itemize}

	\item $x < a$

\[
\begin{aligned}
\lim_{\epsilon \rightarrow 0} 
&\begin{pmatrix}
	\displaystyle{ \frac{a-x}{2t} } \cdot \frac{\begin{aligned}
\frac{e^{- \frac{| u_a |}{\epsilon}} - 1}{{A_\epsilon}\ e^{{A_\epsilon}^2 + \frac{u_a}{\epsilon}}} + \frac{b-x}{a-x} \cdot \frac{1 - e^{- \frac{| u_b |}{\epsilon}}}{{B_\epsilon}\ e^{{B_\epsilon}^2 + \frac{u_a}{\epsilon}}} 
\end{aligned}}{
\begin{aligned}
\sqrt{\pi} + {\textnormal{erfc}} \left( {A_\epsilon} \right)\ e^{\frac{| u_a |}{\epsilon}} \left( 1 - e^{- \frac{| u_a |}{\epsilon}} \right) \\
+ {\textnormal{erfc}} \left( {B_\epsilon} \right)\ e^{\frac{| u_a |}{\epsilon}} \left( e^{- \frac{| u_b |}{\epsilon}} - 1 \right)
\end{aligned} }
		\\		
		\\				
	\frac{
\begin{aligned}
\rho_c\ \textnormal{erfc} \left( B_\epsilon \right)\ e^{\frac{| u_a |}{\epsilon}}\ \left( e^{- \frac{| u_b |}{\epsilon}} - 1 \right) \\
+ \rho_c\ \textnormal{erfc} \left( C_\epsilon \right)\ e^{\frac{| u_a |}{\epsilon}} + \rho_d\ \textnormal{erfc} \left( D_\epsilon \right)\ e^{\frac{| u_a |}{\epsilon}} \cdot e^{- \frac{| u_b |}{\epsilon}}
\end{aligned} 
}{ 
\begin{aligned}
\sqrt{\pi} + \textnormal{erfc} \left( A_\epsilon \right)\ e^{\frac{| u_a |}{\epsilon}}\ \left( 1 - e^{- \frac{| u_a |}{\epsilon}} \right) 
\\
+ \textnormal{erfc} \left( B_\epsilon \right)\ e^{\frac{| u_a |}{\epsilon}}\ \left( e^{- \frac{| u_b |}{\epsilon}} - 1 \right)
\end{aligned} }
\end{pmatrix}^T
\\
	= &\begin{cases}
\left( 0 , 0 \right),\ &{x < a - \sqrt{- 2 u_a t}},
\\
\left( \frac{x-a}{t} , 0 \right),\ &{x > a - \sqrt{- 2 u_a t}}.
\end{cases}
\end{aligned}
\]
The underlying computations for the justification of the above limits (and hence the rest) can be found in the appendix.

	\item $ a < x < c $

\[
\begin{aligned}
\lim_{\epsilon \rightarrow 0} 
&\begin{pmatrix}
	\frac{\epsilon}{\sqrt{2t \epsilon}} \cdot \frac{\begin{aligned}
e^{-{A_\epsilon}^2} \left( e^{-\frac{| u_a |}{\epsilon}} - 1 \right) + e^{-{B_\epsilon}^2} \left( 1 - e^{-\frac{| u_b |}{\epsilon}} \right)
\end{aligned}}{
\begin{aligned}
\sqrt{\pi} + {\textnormal{erfc}} \left( {A_\epsilon} \right) \left( e^{-\frac{| u_a |}{\epsilon}} - 1 \right) + {\textnormal{erfc}} \left( {B_\epsilon} \right) \left( e^{-\frac{| u_b |}{\epsilon}} - 1 \right)
\end{aligned} }
\\
\\
	\frac{
\begin{aligned}
\rho_c\ \textnormal{erfc} \left( B_\epsilon \right) \left( e^{- \frac{| u_b |}{\epsilon}} - 1 \right) 
\\
+ \rho_c\ \textnormal{erfc} \left( C_\epsilon \right) + \rho_d\ \textnormal{erfc} \left( D_\epsilon \right) e^{- \frac{| u_b |}{\epsilon}}
\end{aligned} 
}{ 
\begin{aligned}
\sqrt{\pi} + \textnormal{erfc} \left( A_\epsilon \right) \left( e^{- \frac{| u_a |}{\epsilon}} - 1 \right) + \textnormal{erfc} \left( B_\epsilon \right) \left( e^{- \frac{| u_b |}{\epsilon}} - 1 \right)
\end{aligned} }
\end{pmatrix} 
\\
	= &\left( 0 , 0 \right).	
\end{aligned}
\]

	\item $c < x < b$
	
\[
\begin{aligned}
\lim_{\epsilon \rightarrow 0} 
&\begin{pmatrix}
	\displaystyle{ \frac{\epsilon}{\sqrt{2t \epsilon}} } \cdot \frac{\begin{aligned}
e^{-{A_\epsilon}^2} \left( e^{-\frac{| u_a |}{\epsilon}} - 1 \right) + e^{-{B_\epsilon}^2} \left( 1 - e^{-\frac{| u_b |}{\epsilon}} \right)
\end{aligned}}{
\begin{aligned}
\sqrt{\pi} + {\textnormal{erfc}}({A_\epsilon}) \left( e^{-\frac{| u_a |}{\epsilon}} - 1 \right) + {\textnormal{erfc}}({B_\epsilon}) \left( e^{-\frac{| u_b |}{\epsilon}} - 1 \right)
\end{aligned} }
\\
\\
	\frac{
\begin{aligned}
\sqrt{\pi}\ \rho_c + \rho_c\ \textnormal{erfc} \left( B_\epsilon \right) \left( e^{- \frac{| u_b |}{\epsilon}} - 1 \right) 
\\
- \rho_c\ \textnormal{erfc} \left( C_\epsilon \right) + \rho_d\ \textnormal{erfc} \left( D_\epsilon \right) e^{- \frac{| u_b |}{\epsilon}}
\end{aligned} 
}{ 
\begin{aligned}
\sqrt{\pi} + \textnormal{erfc} \left( A_\epsilon \right) \left( e^{- \frac{| u_a |}{\epsilon}} - 1 \right) + \textnormal{erfc} \left( B_\epsilon \right) \left( e^{- \frac{| u_b |}{\epsilon}} - 1 \right)
\end{aligned} }
\end{pmatrix}^T
\\
		&= ( 0 , \rho_c ).
\end{aligned}
\]

	\item $b < x < d$

\[
\begin{aligned}
\lim_{\epsilon \rightarrow 0} 
&\begin{pmatrix}
	\displaystyle{ \frac{x-a}{2t} } \cdot
\frac{\begin{aligned}
\frac{e^{- \frac{| u_a |}{\epsilon}} - 1}{A_\epsilon\ e^{A_\epsilon^2 - \frac{u_b}{\epsilon}}} + \frac{x-b}{x-a} \cdot \frac{1 - e^{- \frac{| u_b |}{\epsilon}}}{B_\epsilon\ e^{B_\epsilon^2 - \frac{u_b}{\epsilon}}}
\end{aligned} }{\begin{aligned}
\sqrt{\pi} + \textnormal{erfc} \left( A_\epsilon \right) e^{\frac{| u_b |}{\epsilon}} \left( e^{- \frac{| u_b |}{\epsilon}} - 1 \right) + \textnormal{erfc} \left( B_\epsilon \right) e^{\frac{| u_b |}{\epsilon}} \left( 1 - e^{- \frac{| u_b |}{\epsilon}} \right)
\end{aligned} }
		 \\
		 \\
	\frac{
\begin{aligned}
\sqrt{\pi}\ \rho_c + \rho_c\ \textnormal{erfc} \left( B_\epsilon \right) e^{\frac{| u_b |}{\epsilon}} \left( 1 - e^{- \frac{| u_b |}{\epsilon}} \right) 
\\
- \rho_c\ \textnormal{erfc} \left( C_\epsilon \right) e^{\frac{| u_b |}{\epsilon}} + \rho_d\ \textnormal{erfc} \left( D_\epsilon \right)
\end{aligned} 
}{ 
\begin{aligned}
\sqrt{\pi} + \textnormal{erfc} \left( A_\epsilon \right) e^{\frac{| u_b |}{\epsilon}} \left( e^{- \frac{| u_a |}{\epsilon}} - 1 \right) + \textnormal{erfc} \left( B_\epsilon \right) e^{\frac{| u_b |}{\epsilon}} \left( 1 - e^{- \frac{| u_b |}{\epsilon}} \right)
\end{aligned} }		
\end{pmatrix}^T
\\
	= &\begin{cases}
\left( \frac{x-b}{t} , \rho_c \right),\ &x < b + \sqrt{2 u_b t},
\\
\left( 0 , \rho_c \right),\ &x > b + \sqrt{2 u_b t}.
\end{cases}
\end{aligned}
\]

	\item $x > d$

\[
\begin{aligned}
\lim_{\epsilon \rightarrow 0} 
&\begin{pmatrix}
	\displaystyle{ \frac{x-a}{2t} } \cdot
\frac{\begin{aligned}
\frac{e^{- \frac{| u_a |}{\epsilon}} - 1}{A_\epsilon\ e^{A_\epsilon^2 - \frac{u_b}{\epsilon}}} + \frac{x-b}{x-a} \cdot \frac{1 - e^{- \frac{| u_b |}{\epsilon}}}{B_\epsilon\ e^{B_\epsilon^2 - \frac{u_b}{\epsilon}}}
\end{aligned} }{\begin{aligned}
\sqrt{\pi} + \textnormal{erfc} \left( A_\epsilon \right) e^{\frac{| u_b |}{\epsilon}} \left( e^{- \frac{| u_b |}{\epsilon}} - 1 \right) + \textnormal{erfc} \left( B_\epsilon \right) e^{\frac{| u_b |}{\epsilon}} \left( 1 - e^{- \frac{| u_b |}{\epsilon}} \right)
\end{aligned} }
		 \\
		 \\
	\frac{
\begin{aligned}
\sqrt{\pi}\ \left( \rho_c + \rho_d \right) + \rho_c\ \textnormal{erfc} \left( B_\epsilon \right) e^{\frac{| u_b |}{\epsilon}} \left( 1 - e^{- \frac{| u_b |}{\epsilon}} \right) 
\\
- \rho_c\ \textnormal{erfc} \left( C_\epsilon \right) e^{\frac{| u_b |}{\epsilon}} - \rho_d\ \textnormal{erfc} \left( D_\epsilon \right)
\end{aligned} 
}{ 
\begin{aligned}
\sqrt{\pi} + \textnormal{erfc} \left( A_\epsilon \right) e^{\frac{| u_b |}{\epsilon}} \left( e^{- \frac{| u_a |}{\epsilon}} - 1 \right) + \textnormal{erfc} \left( B_\epsilon \right) e^{\frac{| u_b |}{\epsilon}} \left( 1 - e^{- \frac{| u_b |}{\epsilon}} \right)
\end{aligned} }		
\end{pmatrix}^T	 
\\
	= &\begin{cases}
		\left( \frac{x-b}{t} , \rho_c \right),\ &x < b + \sqrt{2 u_b t},
		\\
		\left( 0 , \rho_c + \rho_d \right),\ &x > b + \sqrt{2 u_b t}. 
		\end{cases}
\end{aligned}
\]

\end{itemize}
\end{footnotesize}
To describe our solution, let us introduce the following curves:
\[
\begin{aligned}
\gamma_{{}_a} (t) 
	&:= a - \sqrt{- 2 u_a t},\ \gamma_{{}_b} (t) := b + \sqrt{2 u_b t},\ \gamma_{{}_{c}} (t) := c,\ t \geq 0,
	\\
\gamma_{{}_d} (t) 
	&:= \begin{cases}
d,\ &0 \leq t \leq \frac{(d-b)^2}{2 u_b},
	\\
b + \sqrt{2 u_b t},\ &t \geq \frac{(d-b)^2}{2 u_b}. 			
\end{cases}	    
\end{aligned}
\]
We can now describe our velocity component $u$ as follows:

\[
\begin{aligned}
u(x,t) = \begin{cases}
0,\ &x \in ( -\infty , \gamma_{{}_a} (t) ) \cup (a,b) \cup ( \gamma_{{}_b} (t) , \infty ),
\\
\frac{x-a}{t},\ &x \in ( \gamma_{{}_a} (t) , a ),
\\
\frac{x-b}{t},\ &x \in ( b , \gamma_{{}_b} (t) ).
\end{cases}
\end{aligned}
\]
For obtaining the density component $\rho$, we fix any test function $\phi \in C_c^\infty \left( (-\infty,\infty) \times [0,\infty) \right)$ and compute

\[
\begin{aligned}
	\langle R_x , \phi \rangle 
		&= - \langle R , \phi_x \rangle
	\\
		&= - \int_0^\infty \int_{c}^{\gamma_{{}_d} (t)} \rho_c\ \phi_x\ dx\ dt - \int_0^\infty \int_{\gamma_{{}_d} (t)}^{\infty} ( \rho_c + \rho_d )\ \phi_x\ dx\ dt
	\\
		&= \int_0^\infty \rho_c \left[ \phi \left( c , t \right) - \phi \left( \gamma_{{}_d} (t) , t \right) \right] dt + \int_0^\infty ( \rho_c + \rho_d )\ \phi \left( \gamma_{{}_d} (t) , t \right) dt
	\\
		&= \langle \rho_c \left( \delta_{x = c} - \delta_{x = \gamma_{{}_d} (t)} \right) + ( \rho_c + \rho_d )\ \delta_{x = \gamma_{{}_d} (t)} , \phi \rangle
	\\
		&= \langle \rho_c\ \delta_{x = c} + \rho_d\ \delta_{x = \gamma_{{}_d} (t)} , \phi \rangle,
\end{aligned}
\]
where $R_x$ denotes the distributional derivative of $R = {\displaystyle{\lim_{\epsilon \rightarrow 0}}} R^\epsilon$ w.r.t. $x$. Therefore the density component $\rho$ is given by
\[
\rho = \rho_c\ \delta_{x = \gamma_{{}_{c}} (t)} + \rho_d\ \delta_{x = \gamma_{{}_d} (t)},
\]
where $\gamma_{{}_{c}} : s \longmapsto c$ for every $s \geq 0$. Finally, the graphical representation of our solution is given by
\begin{center}
\includegraphics[scale=0.4]{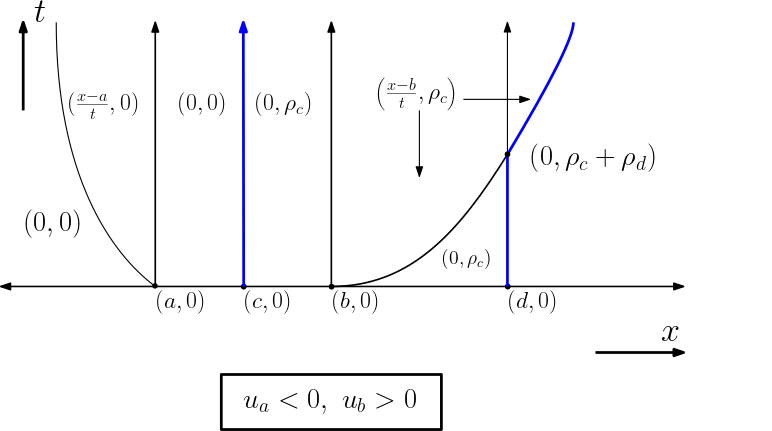}
\end{center}

\textbf{Case 2.} $u_a > 0$, $u_b > 0$

The limit ${\displaystyle{\lim_{\epsilon \rightarrow 0}}} \left( u^\epsilon , R^\epsilon \right)$ is evaluated explicitly in different regions as follows:

\begin{itemize}

	\item $x < a$

\[
\begin{aligned}
\lim_{\epsilon \rightarrow 0} 
&\begin{pmatrix}
	\displaystyle{ \frac{\epsilon}{\sqrt{2 t \epsilon}} } \cdot \frac{\begin{aligned}
e^{- A_\epsilon^2} \left( 1 - e^{- \frac{| u_a |}{\epsilon}} \right) + e^{- B_\epsilon^2} \left( e^{- \frac{| u_a |}{\epsilon}} - e^{- \frac{| u_a + u_b |}{\epsilon}} \right)
\end{aligned}}{\begin{aligned}
\sqrt{\pi} + \textnormal{erfc} \left( A_\epsilon \right) \left( e^{- \frac{| u_a |}{\epsilon}} - 1 \right) + \textnormal{erfc} \left( B_\epsilon \right) \left( e^{- \frac{| u_a + u_b |}{\epsilon}} - e^{- \frac{| u_a |}{\epsilon}} \right)
\end{aligned} }
\\
\\
	\frac{
\begin{aligned}
\rho_c\ \textnormal{erfc} \left( B_\epsilon \right) e^{- \frac{| u_a |}{\epsilon}}\ \left( e^{- \frac{| u_b |}{\epsilon}} - 1 \right) \\
+ \rho_c\ \textnormal{erfc} \left( C_\epsilon \right) e^{- \frac{| u_a |}{\epsilon}} + \rho_d\ \textnormal{erfc} \left( D_\epsilon \right) e^{- \frac{| u_a + u_b |}{\epsilon}}
\end{aligned} 
}{ 
\begin{aligned}
\sqrt{\pi} + \textnormal{erfc} \left( A_\epsilon \right)\ \left( e^{- \frac{| u_a |}{\epsilon}} - 1 \right) + \textnormal{erfc} \left( B_\epsilon \right) e^{- \frac{| u_a |}{\epsilon}} \left( e^{- \frac{| u_b |}{\epsilon}} - 1 \right)
\end{aligned} }								
\end{pmatrix}^T
\\
	= &\left( 0 , 0 \right).
\end{aligned}
\]

	\item $a < x < c$

\[
\begin{aligned}
\lim_{\epsilon \rightarrow 0} 
&\begin{pmatrix}
	\displaystyle{ \frac{x-a}{2t} } \cdot \frac{\begin{aligned}
	\frac{1 - e^{- \frac{| u_a |}{\epsilon}}}{A_\epsilon\ e^{A_\epsilon^2 - \frac{u_a}{\epsilon}}} + \frac{b-x}{x-a} \cdot \frac{1 - e^{- \frac{| u_b |}{\epsilon}}}{B_\epsilon\ e^{B_\epsilon^2}}	
\end{aligned}}{\begin{aligned}
		\sqrt{\pi} + \textnormal{erfc} \left( A_\epsilon \right) e^{\frac{| u_a |}{\epsilon}} \left( 1 - e^{- \frac{| u_a |}{\epsilon}} \right) + \textnormal{erfc} \left( B_\epsilon \right) \left( e^{- \frac{| u_b |}{\epsilon}} - 1 \right)
\end{aligned} }
		\\
		\\
	\frac{
\begin{aligned}
\rho_c\ \textnormal{erfc} \left( B_\epsilon \right)\ \left( e^{- \frac{| u_b |}{\epsilon}} - 1 \right) 
\\
+ \rho_c\ \textnormal{erfc} \left( C_\epsilon \right) + \rho_d\ \textnormal{erfc} \left( D_\epsilon \right) e^{- \frac{| u_b |}{\epsilon}}
\end{aligned} 
}{ 
\begin{aligned}
\sqrt{\pi} + \textnormal{erfc} \left( A_\epsilon \right) e^{\frac{| u_a |}{\epsilon}} \left( 1 - e^{- \frac{| u_a |}{\epsilon}} \right) + \textnormal{erfc} \left( B_\epsilon \right) \left( e^{- \frac{| u_b |}{\epsilon}} - 1 \right)
\end{aligned} }
\end{pmatrix}^T
\\
	= &\begin{cases}
		\left( \frac{x-a}{t} , 0 \right),\ &x < a + \sqrt{2 u_a t},
		\\
		\left( 0 , 0 \right),\ &x > a + \sqrt{2 u_a t}.
	\end{cases}
\end{aligned}
\]

	\item $c < x < b$

\[
\begin{aligned}
\lim_{\epsilon \rightarrow 0} &
\begin{pmatrix}
	\displaystyle{ \frac{x-a}{2t} } \cdot \frac{\begin{aligned}
\frac{1 - e^{- \frac{| u_a |}{\epsilon}}}{A_\epsilon\ e^{A_\epsilon^2 - \frac{u_a}{\epsilon}}} + \frac{b-x}{x-a} \cdot \frac{1 - e^{- \frac{| u_b |}{\epsilon}}}{B_\epsilon\ e^{B_\epsilon^2}}	
\end{aligned}}{\begin{aligned}
\sqrt{\pi} + \textnormal{erfc} \left( A_\epsilon \right) e^{\frac{| u_a |}{\epsilon}} \left( 1 - e^{- \frac{| u_a |}{\epsilon}} \right) + \textnormal{erfc} \left( B_\epsilon \right) \left( e^{- \frac{| u_b |}{\epsilon}} - 1 \right)
\end{aligned} }
		\\
		\\
	\frac{
\begin{aligned}
\sqrt{\pi}\ \rho_c + \rho_c\ \textnormal{erfc} \left( B_\epsilon \right) \left( e^{- \frac{| u_b |}{\epsilon}} - 1 \right) 
\\
- \rho_c\ \textnormal{erfc} \left( C_\epsilon \right) + \rho_d\ \textnormal{erfc} \left( D_\epsilon \right) e^{- \frac{| u_b |}{\epsilon}}
\end{aligned} 
}{ 
\begin{aligned}
\sqrt{\pi} + \textnormal{erfc} \left( A_\epsilon \right) e^{\frac{| u_a |}{\epsilon}} \left( 1 - e^{- \frac{| u_a |}{\epsilon}} \right) + \textnormal{erfc} \left( B_\epsilon \right) \left( e^{- \frac{| u_b |}{\epsilon}} - 1 \right)
\end{aligned} }		
\end{pmatrix}^T
\\
	= &\begin{cases}
		\left( \frac{x-a}{t} , 0 \right),\ &x < a + \sqrt{2 u_a t},
		\\
		\left( 0 , \rho_c \right),\ &x > a + \sqrt{2 u_a t}.
	\end{cases}
\end{aligned}
\]
For the remaining regions, let us consider the curves

\begin{itemize}

	\item $p_a (s) := a + \sqrt{2 \left( u_a + u_b \right) s}$,

	\item $p_b (s) := b + \sqrt{2 u_b s}$,

	\item $l(s) := \frac{a+b}{2} + \frac{u_a}{b-a} s$

\end{itemize}
defined for every $s \geq 0$. Then the structure of our limits can be explicitly described as follows:

\item $b < x < d$

\[
\begin{aligned}
	\lim_{\epsilon \rightarrow 0} 
&\begin{pmatrix}
	\displaystyle{ \frac{x-a}{2t} } \cdot \frac{\begin{aligned}
\frac{1 - e^{- \frac{| u_a |}{\epsilon}}}{A_\epsilon\ e^{A_\epsilon^2 - \frac{u_a + u_b}{\epsilon}}} + \frac{x-b}{x-a} \cdot \frac{1 - e^{- \frac{| u_b |}{\epsilon}}}{B_\epsilon\ e^{B_\epsilon^2 - \frac{u_b}{\epsilon}}}
\end{aligned}}{\begin{aligned}
\sqrt{\pi} + \textnormal{erfc} \left( A_\epsilon \right) e^{\frac{| u_a + u_b |}{\epsilon}} \left( 1 - e^{- \frac{| u_a |}{\epsilon}} \right) + \textnormal{erfc} \left( B_\epsilon \right) \left( 1 - e^{- \frac{| u_b |}{\epsilon}} \right)
\end{aligned} }		
		\\
		\\
	\frac{
\begin{aligned}
\sqrt{\pi}\ \rho_c + \rho_c\ \textnormal{erfc} \left( B_\epsilon \right) e^{\frac{| u_b |}{\epsilon}} \left( 1 - e^{- \frac{| u_b |}{\epsilon}} \right) 
\\
- \rho_c\ \textnormal{erfc} \left( C_\epsilon \right) e^{\frac{| u_b |}{\epsilon}} + \rho_d\ \textnormal{erfc} \left( D_\epsilon \right)
\end{aligned} 
}{ 
\begin{aligned}
\sqrt{\pi} + \textnormal{erfc} \left( A_\epsilon \right) e^{\frac{| u_a + u_b |}{\epsilon}} \left( 1 - e^{- \frac{| u_a |}{\epsilon}} \right) + \textnormal{erfc} \left( B_\epsilon \right) e^{\frac{| u_b |}{\epsilon}} \left( 1 - e^{- \frac{| u_b |}{\epsilon}} \right)
\end{aligned} }
		\end{pmatrix}^T
\\
	= &\begin{cases}
		\left( 0 , \rho_c \right),\ &{ x > \max{ \big\{ p_a (t) , p_b (t) \big\} } },
		\\
		\left( \frac{x-b}{t} , \rho_c \right),\ &{ x \in \Big( p_a (t) , p_b (t) \Big) \cup \Big( l(t) , \min{ \big\{ p_a (t) , p_b (t) \big\} } \Big) },
		\\
		\left( \frac{x-a}{t} , 0 \right),\ &{ x \in \Big( p_b (t) , p_a (t) \Big) \cup \Big( -\infty , \min{ \big\{ p_a (t) , p_b (t) , l(t) \big\} } \Big) }.
	\end{cases}
\end{aligned}
\]

	\item $x > d$

\[
\begin{aligned}
\lim_{\epsilon \rightarrow 0} &
	\begin{pmatrix}
\displaystyle{ \frac{x-a}{2t} } \cdot \frac{\begin{aligned}
\frac{1 - e^{- \frac{| u_a |}{\epsilon}}}{A_\epsilon\ e^{A_\epsilon^2 - \frac{u_a + u_b}{\epsilon}}} + \frac{x-b}{x-a} \cdot \frac{1 - e^{- \frac{| u_b |}{\epsilon}}}{B_\epsilon\ e^{B_\epsilon^2 - \frac{u_b}{\epsilon}}}
\end{aligned}}{\begin{aligned}
\sqrt{\pi} + \textnormal{erfc} \left( A_\epsilon \right) e^{\frac{| u_a + u_b |}{\epsilon}} \left( 1 - e^{- \frac{| u_a |}{\epsilon}} \right) + \textnormal{erfc} \left( B_\epsilon \right) \left( 1 - e^{- \frac{| u_b |}{\epsilon}} \right)
\end{aligned} }
\\
\\
	\frac{
\begin{aligned}
\sqrt{\pi}\ \left( \rho_c + \rho_d \right)+ \rho_c\ \textnormal{erfc} \left( B_\epsilon \right) e^{\frac{| u_b |}{\epsilon}} \left( 1 - e^{- \frac{| u_b |}{\epsilon}} \right) 
\\
- \rho_c\ \textnormal{erfc} \left( C_\epsilon \right) e^{\frac{| u_b |}{\epsilon}} - \rho_d\ \textnormal{erfc} \left( D_\epsilon \right)
\end{aligned} 
}{ 
\begin{aligned}
\sqrt{\pi} + \textnormal{erfc} \left( A_\epsilon \right) e^{\frac{| u_a + u_b |}{\epsilon}} \left( 1 - e^{- \frac{| u_a |}{\epsilon}} \right) + \textnormal{erfc} \left( B_\epsilon \right) e^{\frac{| u_b |}{\epsilon}} \left( 1 - e^{- \frac{| u_b |}{\epsilon}} \right)
\end{aligned} }
\end{pmatrix}^T
\\
	= &\begin{cases}
		\left( 0 , \rho_c + \rho_d \right),\ &{ x > \max{ \big\{ p_a (t) , p_b (t) \big\} } },
		\\
		\left( \frac{x-b}{t} , \rho_c \right),\ &{ x \in \Big( p_a (t) , p_b (t) \Big) \cup \Big( l(t) , \min{ \big\{ p_a (t) , p_b (t) \big\} } \Big) },
		\\
		\left( \frac{x-a}{t} , 0 \right),\ &{ x \in \Big( p_b (t) , p_a (t) \Big) \cup \Big( -\infty , \min{ \big\{ p_a (t) , p_b (t) , l(t) \big\} } \Big) }.   
	\end{cases}
\end{aligned}
\]

\end{itemize}
To describe our solution, we introduce the curves

\[
\begin{aligned}
\gamma_{{}_{a}} (t) &:= \begin{cases}
a + \sqrt{2 u_a t},\ &0 \leq t \leq t_{{}_{a,1}} := \frac{(b-a)^2}{2 u_a},
\\
\frac{a+b}{2} + \frac{u_a}{b-a} t,\ &t_{{}_{a,1}} \leq t \leq t_{{}_{a,2}} := \frac{(d-a)^2}{2 \left( \sqrt{u_a + u_b} - \sqrt{u_b} \right)^2},
\\
a + \sqrt{2 ( u_a + u_b ) t},\ &t \geq t_{{}_{a,2}},
\end{cases}
\\
\\
\gamma_{{}_{b,1}} (t) &:= \begin{cases}
b,\ &0 \leq t \leq t_{{}_{b,1}} := \frac{(b-a)^2}{2 u_a},
\\
\frac{a+b}{2} + \frac{u_a}{b-a} t,\ &t_{{}_{b,1}} \leq t \leq t_{{}_{b,2}} := \frac{(b-a)^2}{2 \left( \sqrt{u_a + u_b} - \sqrt{u_b} \right)^2},
\\
a + \sqrt{2 ( u_a + u_b ) t},\ &t \geq t_{{}_{b,2}},
\end{cases}
\\
\\
\gamma_{{}_{b,2}} (t) &:= \begin{cases}
b + \sqrt{2 u_b t},\ &0 \leq t \leq t_{{}_{b,2}},
\\
a + \sqrt{2 ( u_a + u_b ) t},\ &t \geq t_{{}_{b,2}},
\end{cases}
\\
\\
\gamma_{{}_{c}} (t) &:= \begin{cases}
c,\ &0 \leq t \leq t_{{}_{c,1}} := \frac{(c-a)^2}{2 u_a},
\\
a + \sqrt{2 u_a t},\ &t_{{}_{c,1}} \leq t \leq t_{{}_{c,2}} := \frac{(b-a)^2}{2 u_a},
\\
\frac{a+b}{2} + \frac{u_a}{b-a} t,\ &t_{{}_{c,2}} \leq t \leq t_{{}_{c,3}} := \frac{(d-a)^2}{2 \left( \sqrt{u_a + u_b} - \sqrt{u_b} \right)^2},
\\
a + \sqrt{2 ( u_a + u_b ) t},\ &t \geq t_{{}_{c,3}}.
\end{cases}
\end{aligned}
\]
The velocity component $u$ can now be described as follows:

\[
\begin{aligned}
u(x,t) = \begin{cases}
0,\ &x \in \Big( -\infty , a \Big) \cup \Big( \gamma_{{}_{a}} (t) , \gamma_{{}_{b,1}} (t) \Big) \cup \Big( \gamma_{{}_{b,2}} (t) , \infty \Big),
\\
\frac{x-a}{t},\ &x \in \Big( a , \gamma_{{}_{a}} (t) \Big),
\\
\frac{x-b}{t},\ &x \in \Big( \gamma_{{}_{b,1}} (t) , \gamma_{{}_{b,2}} (t) \Big).
\end{cases}
\end{aligned} 
\]
For getting the density component $\rho$, we have to consider three subcases. First let us consider

\[
\left( x^* , t^* \right) := \left( b + \frac{\left( b-a \right) \sqrt{u_b}}{\sqrt{u_a + u_b} - \sqrt{u_b}} , \frac{(b-a)^2}{2 \left( \sqrt{u_a + u_b} - \sqrt{u_b} \right)^2} \right).
\]
The first subcase is $d < x^*$. Define the curve

\[
\begin{aligned}
\gamma_{{}_{d}} (t) := \begin{cases}
d,\ &0 \leq t \leq t_{{}_{d,1}} := \frac{(d-b)^2}{2 u_b},
\\
b + \sqrt{2 u_b t},\ &t_{{}_{d,1}} \leq t \leq t_{{}_{d,2}} := t^{*},
\\
a + \sqrt{2 ( u_a + u_b ) t},\ &t \geq t_{{}_{d,2}}.
\end{cases}
\end{aligned} 
\]
We observe that any test function $\phi \in C_c^\infty \left( (-\infty,\infty) \times [0,\infty) \right)$ will satisfy

\[
\begin{aligned}
	\langle R_x , \phi \rangle 
		&= - \langle R , \phi_x \rangle
	\\
		&= - \int_0^\infty \int_{\gamma_{{}_{c}} (t)}^{\gamma_{{}_{d}} (t)} \rho_c\ \phi_x\ dx\ dt - \int_0^\infty \int_{\gamma_{{}_{d}} (t)}^{\infty} ( \rho_c + \rho_d )\ \phi_x\ dx\ dt
	\\
		&= \int_0^\infty \rho_c \left[ \phi \left( \gamma_{{}_{c}} (t) , t \right) - \phi \left( \gamma_{{}_{d}} (t) , t \right) \right] dt + \int_0^\infty ( \rho_c + \rho_d )\ \phi \left( \gamma_{{}_{d}} (t) , t \right) dt
	\\
		&= \langle \rho_c \left( \delta_{x = \gamma_{{}_{c}} (t)} - \delta_{x = \gamma_{{}_{d}} (t)} \right) + ( \rho_c + \rho_d )\ \delta_{x = \gamma_{{}_{d}} (t)} , \phi \rangle	
	\\
		&= \langle \rho_c\ \delta_{x = \gamma_{{}_{c}} (t)} + \rho_d\ \delta_{x = \gamma_{{}_{d}} (t)} , \phi \rangle,
\end{aligned}
\]
so that $\rho$ is given by

\[
\rho = \rho_c\ \delta_{x = \gamma_{{}_{c}} (t)} + \rho_d\ \delta_{x = \gamma_{{}_{d}} (t)}.
\]
Under this subcase, our solution is graphically represented as follows:

\begin{center}
\includegraphics[scale=0.4]{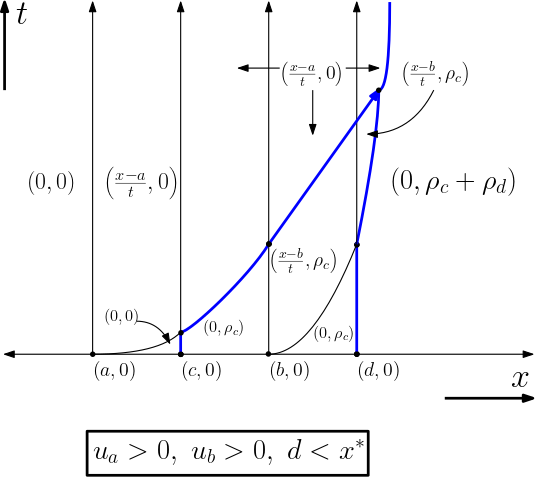}
\end{center}
The second subcase is $d = x^*$. The curve $\gamma_{{}_{d}}$ will be newly defined as

\[
\begin{aligned}
\gamma_{{}_{d}} (t) := \begin{cases}
d,\ &0 \leq t \leq t^*,
\\
a + \sqrt{2 ( u_a + u_b ) t},\ &t \geq t^*,
\end{cases}
\end{aligned} 
\]
and the same computations done for the previous subcase will give us

\[
\rho = \rho_c\ \delta_{x = \gamma_{{}_{c}} (t)} + \rho_d\ \delta_{x = \gamma_{{}_{d}} (t)}.
\]
Under this subcase, our solution is graphically represented as follows:
\begin{center}
\includegraphics[scale=0.4]{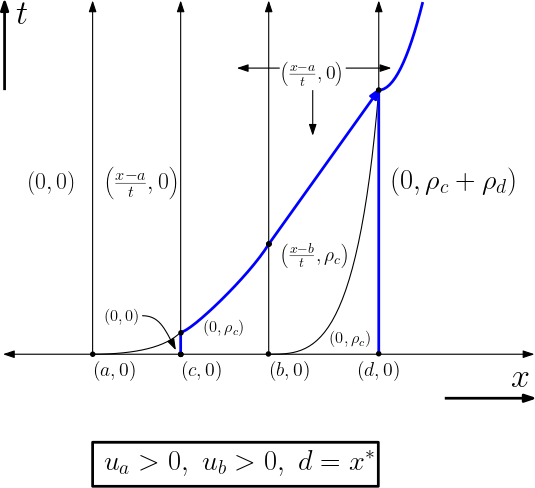}
\end{center}
The only subcase left for us to consider is $d > x^*$. We define
\[
\begin{aligned}
\gamma_{{}_{d}} (t) := \begin{cases}
d,\ &0 \leq t \leq t_{{}_{d}} := \frac{(d-b)^2}{2 u_b},
\\
a + \sqrt{2 ( u_a + u_b ) t},\ &t \geq t_{{}_{d}},
\end{cases}
\end{aligned}
\]
and proceeding along the same lines as in the previous subcases, we obtain

\[
\rho = \rho_c\ \delta_{x = \gamma_{{}_{c}} (t)} + \rho_d\ \delta_{x = \gamma_{{}_{d}} (t)}.
\]
The graphical representation of our solution is given as follows:
\begin{center}
\includegraphics[scale=0.4]{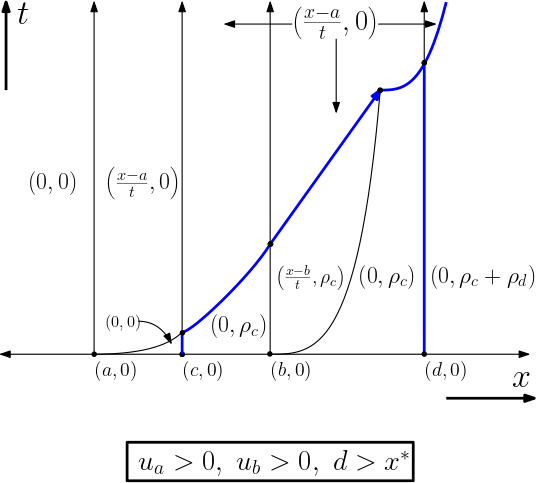}
\end{center}

\textbf{Case 3.} $u_a > 0$, $u_b < 0$, $u_a + u_b > 0$

The limit $\lim_{\epsilon \rightarrow 0} \left( u^\epsilon , R^\epsilon \right)$ is evaluated explicitly in different regions as follows:

\begin{itemize}

	\item $x < a$

\[
\begin{aligned}
\lim_{\epsilon \rightarrow 0} 
&\begin{pmatrix}
	\displaystyle{ \frac{\epsilon}{\sqrt{2 t \epsilon}} } \cdot \frac{\begin{aligned}
\frac{1 - e^{- \frac{| u_a |}{\epsilon}}}{e^{A_\epsilon^2}} + \frac{e^{- \frac{| u_a |}{\epsilon}} - e^{- \frac{| u_a + u_b |}{\epsilon}}}{e^{B_\epsilon^2}}	
	\end{aligned}}{\begin{aligned}
\sqrt{\pi} + \textnormal{erfc} \left( A_\epsilon \right) \left( e^{- \frac{| u_a |}{\epsilon}} - 1 \right) + \textnormal{erfc} \left( e^{- \frac{| u_a + u_b |}{\epsilon}} - e^{- \frac{| u_a |}{\epsilon}} \right)
	\end{aligned} }	
		\\
		\\
	\frac{
\begin{aligned}
\rho_c\ \textnormal{erfc} \left( B_\epsilon \right) e^{- \frac{| u_a + u_b |}{\epsilon}} \left( 1 - e^{- \frac{| u_b |}{\epsilon}} \right) \\
+ \rho_c\ \textnormal{erfc} \left( C_\epsilon \right) e^{- \frac{| u_a |}{\epsilon}} + \rho_d\ \textnormal{erfc} \left( D_\epsilon \right) e^{- \frac{| u_a + u_b |}{\epsilon}}
\end{aligned} 
}{ 
\begin{aligned}
\sqrt{\pi} + \textnormal{erfc} \left( A_\epsilon \right)\ \left( e^{- \frac{| u_a |}{\epsilon}} - 1 \right) + \textnormal{erfc} \left( B_\epsilon \right) e^{- \frac{| u_a + u_b |}{\epsilon}} \left( 1 - e^{- \frac{| u_b |}{\epsilon}} \right)
\end{aligned} }
\end{pmatrix}^T
\\
	= &\left( 0 , 0 \right).
\end{aligned}
\]
Before considering the remaining regions, let us introduce the curves

\begin{itemize}

	\item $p_a (s) := a + \sqrt{2 u_a s}$,

	\item $p_b (s) := b - \sqrt{- 2 u_b s}$,

	\item $l(s) := \frac{a+b}{2} + \frac{u_a + u_b}{b-a} s$

\end{itemize}
defined for every $s \geq 0$. We can now describe the explicit structure of the limit $\lim_{\epsilon \rightarrow 0} \left( u^\epsilon , R^\epsilon \right)$ as follows:

	\item $a < x < c$

\[
\begin{aligned}
\lim_{\epsilon \rightarrow 0} 
&\begin{pmatrix}
\displaystyle{ \frac{x-a}{2t} } \cdot
	\frac{\begin{aligned}
\frac{1 - e^{- \frac{| u_a |}{\epsilon}}}{A_\epsilon\ e^{A_\epsilon^2 - \frac{u_a}{\epsilon}}} + \frac{b-x}{x-a} \cdot \frac{e^{- \frac{| u_b |}{\epsilon}} - 1}{B_\epsilon\ e^{B_\epsilon^2 + \frac{u_b}{\epsilon}}}
\end{aligned}}{\begin{aligned}
	\sqrt{\pi} + \textnormal{erfc} \left( A_\epsilon \right) e^{\frac{| u_a |}{\epsilon}} \left( 1 - e^{- \frac{| u_a |}{\epsilon}} \right) + \textnormal{erfc} \left( B_\epsilon \right) e^{\frac{| u_b |}{\epsilon}} \left( 1 - e^{- \frac{| u_b |}{\epsilon}} \right)
\end{aligned} }	
		\\
		\\
	\frac{
\begin{aligned}
\rho_c\ \textnormal{erfc} \left( B_\epsilon \right) e^{\frac{| u_b |}{\epsilon}} \left( 1 - e^{- \frac{| u_b |}{\epsilon}} \right) 
\\
+ \rho_c\ \textnormal{erfc} \left( C_\epsilon \right) + \rho_d\ \textnormal{erfc} \left( D_\epsilon \right) e^{\frac{| u_b |}{\epsilon}}
\end{aligned} 
}{ 
\begin{aligned}
\sqrt{\pi} + \textnormal{erfc} \left( A_\epsilon \right) e^{\frac{| u_a |}{\epsilon}} \left( 1 - e^{- \frac{|
u_a |}{\epsilon}} \right) + \textnormal{erfc} \left( B_\epsilon \right) e^{\frac{| u_b |}{\epsilon}} \left( 1 - e^{- \frac{| u_b |}{\epsilon}} \right)
\end{aligned} }	
\end{pmatrix}^T
\\
	= &\begin{cases}
		\left( 0 , 0 \right),\ &{ x \in \Big( p_a (t) , p_b (t) \Big) },
		\\
		\left( \frac{x-b}{t} , \rho_c \right),\ &{ x \in \Big( \max{ \big\{ p_a (t) , p_b (t) \big\} } , \infty \Big) \cup \Big( \max{ \big\{ p_b (t) , l(t) \big\} } , p_a (t) \Big) },
		\\
		\left( \frac{x-a}{t} , 0 \right),\ &{ x \in \Big( -\infty , \min{ \big\{ p_a (t) , p_b (t) \big\} } \Big) \cup \Big( p_b (t) , \min{ \big\{ p_a (t) , l(t) \big\} } \Big) }.
		\end{cases}
\end{aligned}
\]

	\item $c < x < b$

\[
\begin{aligned}
\lim_{\epsilon \rightarrow 0} 
&\begin{pmatrix}
	\displaystyle{ \frac{x-a}{2t} } \cdot \frac{\begin{aligned}
	\frac{1 - e^{- \frac{| u_a |}{\epsilon}}}{A_\epsilon\ e^{A_\epsilon^2 - \frac{u_a}{\epsilon}}} + \frac{b-x}{x-a} \cdot \frac{e^{- \frac{| u_b |}{\epsilon}} - 1}{B_\epsilon\ e^{B_\epsilon^2 + \frac{u_b}{\epsilon}}}
\end{aligned}}{\begin{aligned}
	\sqrt{\pi} + \textnormal{erfc} \left( A_\epsilon \right) e^{\frac{| u_a |}{\epsilon}} \left( 1 - e^{- \frac{| u_a |}{\epsilon}} \right) + \textnormal{erfc} \left( B_\epsilon \right) e^{\frac{| u_b |}{\epsilon}} \left( 1 - e^{- \frac{| u_b |}{\epsilon}} \right)
\end{aligned} }	
		\\
		\\
	\frac{
\begin{aligned}
\sqrt{\pi}\ \rho_c + \rho_c\ \textnormal{erfc} \left( B_\epsilon \right) e^{\frac{| u_b |}{\epsilon}} \left( 1 - e^{- \frac{| u_b |}{\epsilon}} \right) 
\\
- \rho_c\ \textnormal{erfc} \left( C_\epsilon \right) + \rho_d\ \textnormal{erfc} \left( D_\epsilon \right) e^{\frac{| u_b |}{\epsilon}}
\end{aligned} 
}{ 
\begin{aligned}
\sqrt{\pi} + \textnormal{erfc} \left( A_\epsilon \right) e^{\frac{| u_a |}{\epsilon}} \left( 1 - e^{- \frac{| u_a |}{\epsilon}} \right) + \textnormal{erfc} \left( B_\epsilon \right) e^{\frac{| u_b |}{\epsilon}} \left( 1 - e^{- \frac{| u_b |}{\epsilon}} \right)
\end{aligned} }
\end{pmatrix}^T		 	
\\
	= &\begin{cases}
		\left( 0 , \rho_c \right),\ &{ x \in \Big( p_a (t) , p_b (t) \Big) },
		\\
		\left( \frac{x-b}{t} , \rho_c \right),\ &{ x \in \Big( \max{ \big\{ p_a (t) , p_b (t) \big\} } , \infty \Big) \cup \Big( \max{ \big\{ p_b (t) , l(t) \big\} } , p_a (t) \Big) },
		\\
		\left( \frac{x-a}{t} , 0 \right),\ &{ x \in \Big( -\infty , \min{ \big\{ p_a (t) , p_b (t) \big\} } \Big) \cup \Big( p_b (t) , \min{ \big\{ p_a (t) , l(t) \big\} } \Big) }.
		\end{cases}
\end{aligned}
\]

	\item $b < x < d$

\[
\begin{aligned}
\lim_{\epsilon \rightarrow 0} 
&\begin{pmatrix}
	\displaystyle{ \frac{x-a}{2t} } \cdot \frac{\begin{aligned}
\frac{1 - e^{- \frac{| u_a |}{\epsilon}}}{A_\epsilon\ e^{A_\epsilon^2 - \frac{u_a + u_b}{\epsilon}}} + \frac{x-b}{x-a} \cdot \frac{e^{- \frac{| u_b |}{\epsilon}} - 1}{B_\epsilon\ e^{B_\epsilon^2}}	
\end{aligned}}{\begin{aligned}
	\sqrt{\pi} + \textnormal{erfc} \left( A_\epsilon \right) e^{\frac{| u_a + u_b |}{\epsilon}} \left( 1 - e^{- \frac{| u_a |}{\epsilon}} \right) + \textnormal{erfc} \left( B_\epsilon \right) \left( e^{- \frac{| u_a |}{\epsilon}} - 1 \right)
\end{aligned} }
		\\
		\\
	\frac{
\begin{aligned}
\sqrt{\pi}\ \rho_c + \rho_c\ \textnormal{erfc} \left( B_\epsilon \right) \left( e^{- \frac{| u_b |}{\epsilon}} - 1 \right) 
\\
- \rho_c\ \textnormal{erfc} \left( C_\epsilon \right) e^{- \frac{| u_b |}{\epsilon}} + \rho_d\ \textnormal{erfc} \left( D_\epsilon \right)
\end{aligned} 
}{ 
\begin{aligned}
\sqrt{\pi} + \textnormal{erfc} \left( A_\epsilon \right) e^{\frac{| u_a + u_b |}{\epsilon}} \left( 1 - e^{- \frac{| u_a |}{\epsilon}} \right) + \textnormal{erfc} \left( B_\epsilon \right) \left( e^{- \frac{| u_b |}{\epsilon}} - 1 \right)
\end{aligned} }							
\end{pmatrix}^T
\\
	= &\begin{cases}
		\left( 0 , \rho_c \right),\ &x > a + \sqrt{2 ( u_a + u_b ) t},
		\\
		\left( \frac{x-a}{t} , 0 \right),\ &x < a + \sqrt{2 ( u_a + u_b ) t}.
		\end{cases}
\end{aligned}
\]

	\item $x > d$

\[
\begin{aligned}
\lim_{\epsilon \rightarrow 0} 
&\begin{pmatrix}
	\displaystyle{ \frac{x-a}{2t} } \cdot \frac{\begin{aligned}
\frac{1 - e^{- \frac{| u_a |}{\epsilon}}}{A_\epsilon\ e^{A_\epsilon^2 - \frac{u_a + u_b}{\epsilon}}} + \frac{x-b}{x-a} \cdot \frac{e^{- \frac{| u_b |}{\epsilon}} - 1}{B_\epsilon\ e^{B_\epsilon^2}}	
\end{aligned}}{\begin{aligned}
	\sqrt{\pi} + \textnormal{erfc} \left( A_\epsilon \right) e^{\frac{| u_a + u_b |}{\epsilon}} \left( 1 - e^{- \frac{| u_a |}{\epsilon}} \right) + \textnormal{erfc} \left( B_\epsilon \right) \left( e^{- \frac{| u_a |}{\epsilon}} - 1 \right)
\end{aligned} }
		\\
		\\
	\frac{
\begin{aligned}
\sqrt{\pi}\ \left( \rho_c + \rho_d \right) + \rho_c\ \textnormal{erfc} \left( B_\epsilon \right) \left( e^{- \frac{| u_b |}{\epsilon}} - 1 \right) 
\\
- \rho_c\ \textnormal{erfc} \left( C_\epsilon \right) e^{- \frac{| u_b |}{\epsilon}} - \rho_d\ \textnormal{erfc} \left( D_\epsilon \right)
\end{aligned} 
}{ 
\begin{aligned}
\sqrt{\pi} + \textnormal{erfc} \left( A_\epsilon \right) e^{\frac{| u_a + u_b |}{\epsilon}} \left( 1 - e^{- \frac{| u_a |}{\epsilon}} \right) + \textnormal{erfc} \left( B_\epsilon \right) \left( e^{- \frac{| u_b |}{\epsilon}} - 1 \right)
\end{aligned} }	
\end{pmatrix}^T				
\\
	= &\begin{cases}
		\left( 0 , \rho_c + \rho_d \right),\ &x > a + \sqrt{2 ( u_a + u_b ) t},
		\\
		\left( \frac{x-a}{t} , 0 \right),\ &x < a + \sqrt{2 ( u_a + u_b ) t}.
		\end{cases}
\end{aligned}
\]

\end{itemize}
To describe our solution, let us define the following curves:
\[
\begin{aligned}
\gamma_{{}_{a}} (t) &:= \begin{cases}
a + \sqrt{2 u_a t},\ &0 \leq t \leq t_{{}_{a,1}} := \frac{(b-a)^2}{2 \left( \sqrt{u_a} + \sqrt{- u_b} \right)^2},
\\
\frac{a+b}{2} + \frac{u_a + u_b}{b-a} t,\ &t_{{}_{a,1}} \leq t \leq t_{{}_{a,2}} := \frac{(b-a)^2}{2 ( u_a + u_b )},
\\
a + \sqrt{2 ( u_a + u_b ) t},\ &t \geq t_{{}_{a,2}},
\end{cases}
\\
\\
\gamma_{{}_{b,1}} (t) &:= \begin{cases}
b - \sqrt{- 2 u_b t},\ &0 \leq t \leq t_{{}_{b,1}} := \frac{(b-a)^2}{2 \left( \sqrt{u_a} + \sqrt{- u_b} \right)^2},
\\
\frac{a+b}{2} + \frac{u_a + u_b}{b-a} t,\ &t_{{}_{b,1}} \leq t \leq t_{{}_{b,2}} := \frac{(b-a)^2}{2 ( u_a + u_b )},
\\
a + \sqrt{2 ( u_a + u_b ) t},\ &t \geq t_{{}_{b,2}},
\end{cases}
\\
\\
\gamma_{{}_{b,2}} (t) &:= \begin{cases}
b,\ &0 \leq t \leq t_{{}_{b,2}},
\\
a + \sqrt{2 ( u_a + u_b ) t},\ &t \geq t_{{}_{b,2}}.
\end{cases}
\end{aligned}
\]
The velocity component $u$ can now be described as follows:

\[
\begin{aligned}
u(x,t) = \begin{cases}
0,\ &x \in \left( -\infty , a \right) \cup \left( \gamma_{{}_{a}} (t) , \gamma_{{}_{b,1}} (t) \right) \cup \left( \gamma_{{}_{b,2}} (t) , \infty \right),
\\
\frac{x-a}{t},\ &x \in \left( a , \gamma_{{}_{a}} (t) \right) ,
\\
\frac{x-b}{t},\ &x \in \left( \gamma_{{}_{b,1}} (t) , \gamma_{{}_{b,2}} (t) \right) .
\end{cases}
\end{aligned}
\]
For describing the density component $\rho$, let us first define
\[
\begin{aligned}
\gamma_{{}_{d}} (t) := \begin{cases}
d,\ &0 \leq t \leq t_{{}_{d,1}} := \frac{(d-a)^2}{2 ( u_a + u_b )},
\\
a + \sqrt{2 ( u_a + u_b ) t},\ &t \geq t_{{}_{d,1}}.
\end{cases}
\end{aligned}
\]
To proceed further, we have to consider three subcases. For this purpose, we define

\[
\begin{aligned}
\left( x^* , t^* \right) := \left( a + \frac{\left( b-a \right) \sqrt{u_a}}{\sqrt{u_a} + \sqrt{- u_b}} , \frac{(b-a)^2}{2 \left( \sqrt{u_a} + \sqrt{- u_b} \right)^2} \right).
\end{aligned}
\]
The first subcase is $c < x^*$. Define
\[
\begin{aligned}
\gamma_{{}_{c}} (t) := \begin{cases}
c,\ &0 \leq t \leq t_{{}_{c,1}} := \frac{(c-a)^2}{2 u_a},
\\
a + \sqrt{2 u_a t},\ &t_{{}_{c,1}} \leq t \leq t_{{}_{c,2}} := t^*,
\\
\frac{a+b}{2} + \frac{u_a + u_b}{b-a} t,\ &t_{{}_{c,2}} \leq t \leq t_{{}_{c,3}} := \frac{(b-a)^2}{2 ( u_a + u_b )},
\\
a + \sqrt{2 ( u_a + u_b ) t},\ &t \geq t_{{}_{c,3}}.
\end{cases}
\end{aligned}
\]
Now we observe that any test function $\phi \in C_c^\infty \left( (-\infty,\infty) \times [0,\infty) \right)$ satisfies

\[
\begin{aligned}
	\langle R_x , \phi \rangle 
		&= - \langle R , \phi_x \rangle
	\\
		&= - \int_0^\infty \int_{\gamma_c (t)}^{\gamma_d (t)} \rho_c\ \phi_x\ dx\ dt - \int_0^\infty \int_{\gamma_d (t)}^{\infty} ( \rho_c + \rho_d )\ \phi_x\ dx\ dt
	\\
		&= \int_0^\infty \rho_c \left[ \phi \left( \gamma_c (t) , t \right) - \phi \left( \gamma_d (t) , t \right) \right] dt + \int_0^\infty ( \rho_c + \rho_d )\ \phi \left( \gamma_d (t) , t \right)\ dt
	\\
		&= \langle \rho_c \left( \delta_{x = \gamma_c (t)} - \delta_{x = \gamma_d (t)} \right) + ( \rho_c + \rho_d )\ \delta_{x = \gamma_d (t)} , \phi \rangle
	\\
		&= \langle \rho_c\ \delta_{x = \gamma_c (t)} + \rho_d\ \delta_{x = \gamma_d (t)} , \phi \rangle,
\end{aligned}
\]
so that $\rho$ is given by

\[
\rho = \rho_c\ \delta_{x = \gamma_{{}_{c}} (t)} + \rho_d\ \delta_{x = \gamma_{{}_{d}} (t)}.
\]
Our solution can now be graphically represented as follows:

\begin{center}
\includegraphics[scale=0.3]{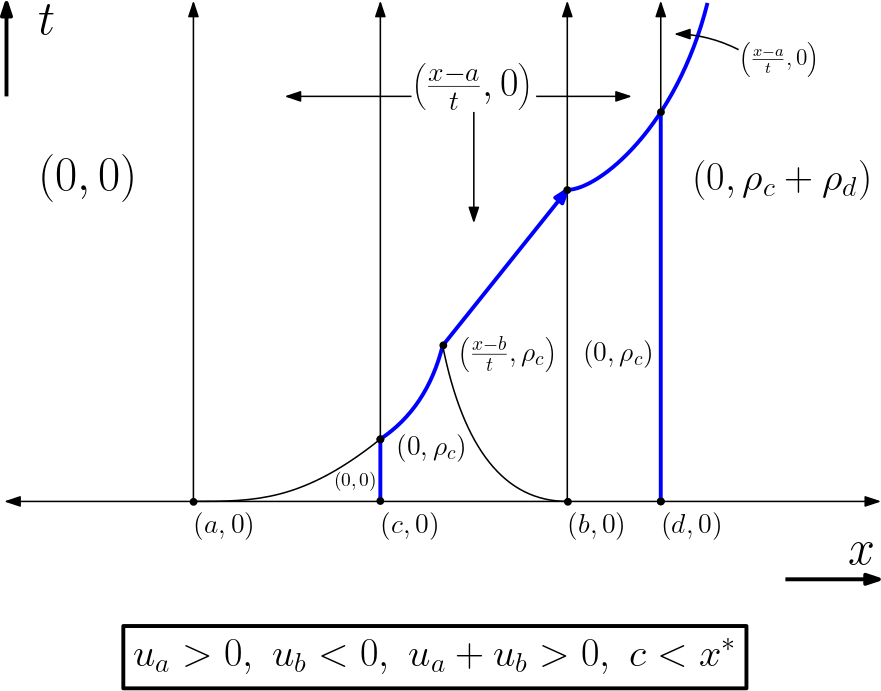}
\end{center}

The next subcase is $c = x^*$. Defining
\[
\begin{aligned}
\gamma_{{}_{c}} (t) := \begin{cases}
c,\ &0 \leq t \leq t_{{}_{c,1}} := t^*,
\\
\frac{a+b}{2} + \frac{u_a + u_b}{b-a} t,\ &t_{{}_{c,1}} \leq t \leq t_{{}_{c,2}} := \frac{(b-a)^2}{2 ( u_a + u_b )},
\\
a + \sqrt{2 ( u_a + u_b ) t},\ &t \geq t_{{}_{c,2}},
\end{cases}
\end{aligned}
\]
we can follow the same procedure as in the previous subcase leading us to the conclusion

\[
\rho = \rho_c\ \delta_{x = \gamma_{{}_{c}} (t)} + \rho_d\ \delta_{x = \gamma_{{}_{d}} (t)}.
\]
The graphical representation of our solution under this subcase is given as follows:

\begin{center}
\includegraphics[scale=0.3]{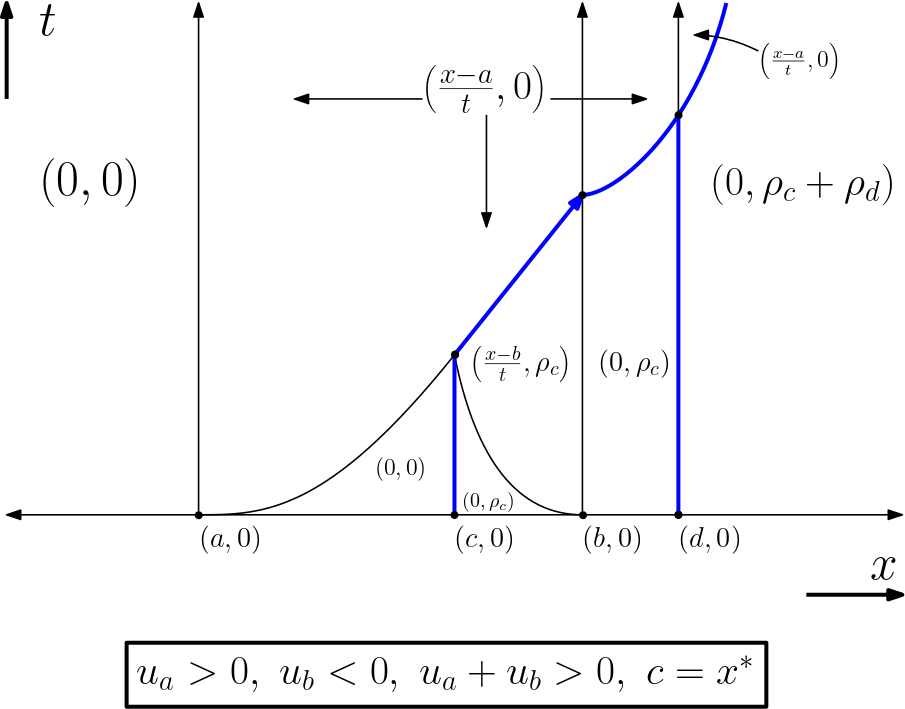}
\end{center}

The final subcase to be considered is $c > x^*$. We define

\[
\begin{aligned}
\gamma_{{}_{c}} (t) := \begin{cases}
c,\ &0 \leq t \leq t_{{}_{c,1}} := \frac{(b-c)^2}{- 2 u_b},
\\
b - \sqrt{- 2 u_b t},\ &t_{{}_{c,1}} \leq t \leq t_{{}_{c,2}} := t^*,
\\
\frac{a+b}{2} + \frac{u_a + u_b}{b-a} t,\ &t_{{}_{c,2}} \leq t \leq t_{{}_{c,3}} := \frac{(b-a)^2}{2 ( u_a + u_b )},
\\
a + \sqrt{2 ( u_a + u_b ) t},\ &t \geq t_{{}_{c,3}}.
\end{cases}
\end{aligned}
\]
The same computations as done in the previous subcases show that

\[
\rho = \rho_c\ \delta_{x = \gamma_{{}_{c}} (t)} + \rho_d\ \delta_{x = \gamma_{{}_{d}} (t)}.
\]
The graphical representation of our solution under this subcase is given as follows:

\begin{center}
\includegraphics[scale=0.6]{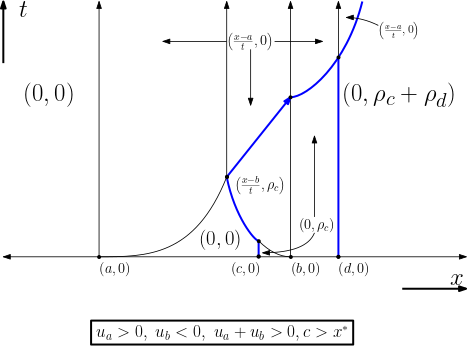}
\end{center}

\textbf{Case 4.} $u_a > 0$, $u_b < 0$, $u_a + u_b < 0$

The limit $\lim_{\epsilon \rightarrow 0} \left( u^\epsilon , R^\epsilon \right)$ is evaluated explicitly in different regions as follows:

\begin{itemize}

	\item $x < a$

\[
\begin{aligned}
\lim_{\epsilon \rightarrow 0} 
&\begin{pmatrix}
	\displaystyle{ \frac{b-x}{2t} } \cdot \frac{\begin{aligned}
\frac{a-x}{b-x} \cdot \frac{1 - e^{- \frac{| u_a |}{\epsilon}}}{A_\epsilon\ e^{A_\epsilon^2}} + \frac{e^{- \frac{| u_b |}{\epsilon}} - 1}{B_\epsilon\ e^{B_\epsilon^2 + \frac{u_a + u_b }{\epsilon}}}
\end{aligned}}{\begin{aligned}
	\sqrt{\pi} + \textnormal{erfc} \left( A_\epsilon \right) \left( e^{- \frac{| u_a |}{\epsilon}} - 1 \right) + \textnormal{erfc} \left(B_\epsilon \right) e^{\frac{| u_a + u_b |}{\epsilon}} \left( 1 - e^{- \frac{| u_b |}{\epsilon}} \right)
\end{aligned} }
		\\
		\\
	\frac{
\begin{aligned}
\rho_c\ \textnormal{erfc} \left( B_\epsilon \right) e^{\frac{| u_a + u_b |}{\epsilon}} \left( 1 - e^{- \frac{| u_b |}{\epsilon}} \right) \\
+ \rho_c\ \textnormal{erfc} \left( C_\epsilon \right) e^{- \frac{| u_a |}{\epsilon}} + \rho_d\ \textnormal{erfc} \left( D_\epsilon \right) e^{\frac{| u_a + u_b |}{\epsilon}}
\end{aligned} 
}{ 
\begin{aligned}
\sqrt{\pi} + \textnormal{erfc} \left( A_\epsilon \right)\ \left( e^{- \frac{| u_a |}{\epsilon}} - 1 \right) + \textnormal{erfc} \left( B_\epsilon \right) e^{\frac{| u_a + u_b |}{\epsilon}} \left( 1 - e^{- \frac{| u_b |}{\epsilon}} \right)
\end{aligned} }				
\end{pmatrix}^T
\\
	= &\begin{cases}
		\left( 0 , 0 \right),\ &x < b - \sqrt{- 2 ( u_a + u_b ) t},
		\\
		\left( \frac{x-b}{t} , \rho_c \right),\ &x > b - \sqrt{- 2 ( u_a + u_b ) t},
		\end{cases}
\end{aligned}
\]
Next, we introduce the curves

\begin{itemize}

	\item $p_a (s) := a + \sqrt{2 u_a s}$,

	\item $p_b (s) := b - \sqrt{- 2 u_b s}$,

	\item $l(s) := \frac{a+b}{2} + \frac{u_a + u_b}{b-a} s$

\end{itemize}
defined for every $s \geq 0$. The vanishing viscosity limits can then be described as follows:

	\item $a < x < c$

\[
\begin{aligned}
\lim_{\epsilon \rightarrow 0} 
&\begin{pmatrix}
	\displaystyle{ \frac{x-a}{2t} } \cdot \frac{\begin{aligned}
	\frac{1 - e^{- \frac{| u_a |}{\epsilon}}}{A_\epsilon\ e^{A_\epsilon^2 - \frac{u_a}{\epsilon}}} + \frac{b-x}{x-a} \cdot \frac{e^{- \frac{| u_b |}{\epsilon}} - 1}{B_\epsilon\ e^{B_\epsilon^2 + \frac{u_b}{\epsilon}}} 
\end{aligned}}{\begin{aligned}
	\sqrt{\pi} + \textnormal{erfc} \left( A_\epsilon \right) e^{\frac{| u_a |}{\epsilon}} \left( 1 - e^{- \frac{| u_a |}{\epsilon}} \right) + \textnormal{erfc} \left( B_\epsilon \right) e^{\frac{| u_b |}{\epsilon}} \left( 1 - e^{- \frac{| u_b |}{\epsilon}} \right)
\end{aligned}	 } 
		\\
		\\
	\frac{
\begin{aligned}
\rho_c\ \textnormal{erfc} \left( B_\epsilon \right) e^{\frac{| u_b |}{\epsilon}} \left( 1 - e^{- \frac{| u_b |}{\epsilon}} \right) 
\\
+ \rho_c\ \textnormal{erfc} \left( C_\epsilon \right) + \rho_d\ \textnormal{erfc} \left( D_\epsilon \right) e^{\frac{| u_b |}{\epsilon}}
\end{aligned} 
}{ 
\begin{aligned}
\sqrt{\pi} + \textnormal{erfc} \left( A_\epsilon \right) e^{\frac{| u_a |}{\epsilon}} \left( 1 - e^{- \frac{| u_a |}{\epsilon}} \right) + \textnormal{erfc} \left( B_\epsilon \right) e^{\frac{| u_b |}{\epsilon}} \left( 1 - e^{- \frac{| u_b |}{\epsilon}} \right)
\end{aligned} }									
\end{pmatrix}^T		 		
\\
	= &\begin{cases}
		\left( 0 , 0 \right),\ &{ x \in \Big( p_a (t) , p_b (t) \Big) },
		\\
		\left( \frac{x-b}{t} , \rho_c \right),\ &{ x \in \Big( \max{ \big\{ p_a (t) , p_b (t) \big\} } , \infty \Big) \cup \Big( \max{ \big\{ p_b (t) , l(t) \big\} } , p_a (t) \Big) },
		\\
		\left( \frac{x-a}{t} , 0 \right),\ &{ x \in \Big( -\infty , \min{ \big\{ p_a (t) , p_b (t) \big\} } \Big) \cup \Big( p_b (t) , \min{ \big\{ p_a (t) , l(t) \big\} } \Big) }.
		\end{cases}
\end{aligned}
\]
	
	\item $c < x < b$

\[
\begin{aligned}
\lim_{\epsilon \rightarrow 0} 
&\begin{pmatrix}
	\displaystyle{ \frac{x-a}{2t} } \cdot \frac{\begin{aligned}
	\frac{1 - e^{- \frac{| u_a |}{\epsilon}}}{A_\epsilon\ e^{A_\epsilon^2 - \frac{u_a}{\epsilon}}} + \frac{b-x}{x-a} \cdot \frac{e^{- \frac{| u_b |}{\epsilon}} - 1}{B_\epsilon\ e^{B_\epsilon^2 + \frac{u_b}{\epsilon}}} 
\end{aligned}}{\begin{aligned}
	\sqrt{\pi} + \textnormal{erfc} \left( A_\epsilon \right) e^{\frac{| u_a |}{\epsilon}} \left( 1 - e^{- \frac{| u_a |}{\epsilon}} \right) + \textnormal{erfc} \left( B_\epsilon \right) e^{\frac{| u_b |}{\epsilon}} \left( 1 - e^{- \frac{| u_b |}{\epsilon}} \right)
\end{aligned}	 } 
		\\
		\\
	\frac{
\begin{aligned}
\sqrt{\pi}\ \rho_c + \rho_c\ \textnormal{erfc} \left( B_\epsilon \right) e^{\frac{| u_b |}{\epsilon}} \left( 1 - e^{- \frac{| u_b |}{\epsilon}} \right) 
\\
- \rho_c\ \textnormal{erfc} \left( C_\epsilon \right) + \rho_d\ \textnormal{erfc} \left( D_\epsilon \right) e^{\frac{| u_b |}{\epsilon}}
\end{aligned} 
}{ 
\begin{aligned}
\sqrt{\pi} + \textnormal{erfc} \left( A_\epsilon \right) e^{\frac{| u_a |}{\epsilon}} \left( 1 - e^{- \frac{| u_a |}{\epsilon}} \right) + \textnormal{erfc} \left( B_\epsilon \right) e^{\frac{| u_b |}{\epsilon}} \left( 1 - e^{- \frac{| u_b |}{\epsilon}} \right)
\end{aligned} }	
\end{pmatrix}^T
\\
	= &\begin{cases}
		\left( 0 , \rho_c \right),\ &{ x \in \Big( p_a (t) , p_b (t) \Big) },
		\\
		\left( \frac{x-b}{t} , \rho_c \right),\ &{ x \in \Big( \max{ \big\{ p_a (t) , p_b (t) \big\} } , \infty \Big) \cup \Big( \max{ \big\{ p_b (t) , l(t) \big\} } , p_a (t) \Big) },
		\\
		\left( \frac{x-a}{t} , 0 \right),\ &{ x \in \Big( -\infty , \min{ \big\{ p_a (t) , p_b (t) \big\} } \Big) \cup \Big( p_b (t) , \min{ \big\{ p_a (t) , l(t) \big\} } \Big) }.
		\end{cases}
\end{aligned}
\]

	\item $b < x < d$

\[
\begin{aligned}
\lim_{\epsilon \rightarrow 0} 
&\begin{pmatrix}
	\displaystyle{ \frac{\epsilon}{\sqrt{2 t \epsilon}} } \cdot \frac{\begin{aligned}
\frac{\left( 1 - e^{- \frac{| u_a |}{\epsilon}} \right) e^{- \frac{| u_a + u_b |}{\epsilon}}}{e^{A_\epsilon^2}} + \frac{e^{- \frac{| u_b |}{\epsilon}} - 1}{e^{B_\epsilon^2}}
\end{aligned}}{\begin{aligned}
	\sqrt{\pi} + \textnormal{erfc} \left( A_\epsilon \right) e^{- \frac{| u_a + u_b |}{\epsilon}} \left( 1 - e^{- \frac{| u_a |}{\epsilon}} \right) + \textnormal{erfc} \left( B_\epsilon \right) \left( e^{- \frac{| u_b |}{\epsilon}} - 1 \right)
\end{aligned} }
		\\
		\\
	\frac{
\begin{aligned}
\sqrt{\pi}\ \rho_c + \rho_c\ \textnormal{erfc} \left( B_\epsilon \right) \left( e^{- \frac{| u_b |}{\epsilon}} - 1 \right) 
\\
- \rho_c\ \textnormal{erfc} \left( C_\epsilon \right) e^{- \frac{| u_b |}{\epsilon}} + \rho_d\ \textnormal{erfc} \left( D_\epsilon \right)
\end{aligned} 
}{ 
\begin{aligned}
\sqrt{\pi} + \textnormal{erfc} \left( A_\epsilon \right) e^{- \frac{| u_a + u_b |}{\epsilon}} \left( 1 - e^{- \frac{| u_a |}{\epsilon}} \right) + \textnormal{erfc} \left( B_\epsilon \right) \left( e^{- \frac{| u_b |}{\epsilon}} - 1 \right)
\end{aligned} }	
\end{pmatrix}^T
\\
	= &\left( 0 , \rho_c \right).
\end{aligned}
\]

	\item $x > d$

\[
\begin{aligned}
\lim_{\epsilon \rightarrow 0} 
&\begin{pmatrix}
	\displaystyle{ \frac{\epsilon}{\sqrt{2 t \epsilon}} } \cdot \frac{\begin{aligned}
\frac{\left( 1 - e^{- \frac{| u_a |}{\epsilon}} \right) e^{- \frac{| u_a + u_b |}{\epsilon}}}{e^{A_\epsilon^2}} + \frac{e^{- \frac{| u_b |}{\epsilon}} - 1}{e^{B_\epsilon^2}}
\end{aligned}}{\begin{aligned}
	\sqrt{\pi} + \textnormal{erfc} \left( A_\epsilon \right) e^{- \frac{| u_a + u_b |}{\epsilon}} \left( 1 - e^{- \frac{| u_a |}{\epsilon}} \right) + \textnormal{erfc} \left( B_\epsilon \right) \left( e^{- \frac{| u_b |}{\epsilon}} - 1 \right)
\end{aligned} }
		\\
		\\
	\frac{
\begin{aligned}
\sqrt{\pi}\ \left( \rho_c + \rho_d \right) + \rho_c\ \textnormal{erfc} \left( B_\epsilon \right) \left( e^{- \frac{| u_b |}{\epsilon}} - 1 \right) 
\\
- \rho_c\ \textnormal{erfc} \left( C_\epsilon \right) e^{- \frac{| u_b |}{\epsilon}} - \rho_d\ \textnormal{erfc} \left( D_\epsilon \right)
\end{aligned} 
}{ 
\begin{aligned}
\sqrt{\pi} + \textnormal{erfc} \left( A_\epsilon \right) e^{- \frac{| u_a + u_b |}{\epsilon}} \left( 1 - e^{- \frac{| u_a |}{\epsilon}} \right) + \textnormal{erfc} \left( B_\epsilon \right) \left( e^{- \frac{| u_b |}{\epsilon}} - 1 \right)
\end{aligned} }	
\end{pmatrix}^T		 	
\\
	= &\left( 0 , \rho_c + \rho_d \right).
\end{aligned}
\]

\end{itemize}
For describing our solution, let us introduce the following curves:

\[
\begin{aligned}
\gamma_{{}_{a,1}} (t) &:= \begin{cases}
a,\ &0 \leq t \leq t_{{}_{a,2}} := \frac{(b-a)^2}{- 2 ( u_a + u_b )},
\\
b - \sqrt{- 2 ( u_a + u_b ) t},\ &t \geq t_{{}_{a,2}},
\end{cases}
\\
\\
\gamma_{{}_{a,2}} (t) &:= \begin{cases}
a + \sqrt{2 u_a t},\ &0 \leq t \leq t_{{}_{a,1}} := \frac{(b-a)^2}{2 \left( \sqrt{u_a} + \sqrt{- u_b} \right)^2},
\\
\frac{a+b}{2} + \frac{u_a + u_b}{b-a} t,\ &t_{{}_{a,1}} \leq t \leq t_{{}_{a,2}},
\\
b - \sqrt{- 2 ( u_a + u_b ) t},\ &t \geq t_{{}_{a,2}},
\end{cases}
\\
\\
\gamma_{{}_{b}} (t) &:= \begin{cases}
b - \sqrt{- 2 u_b t},\ &0 \leq t \leq t_{{}_{b,1}} := \frac{(b-a)^2}{2 \left( \sqrt{u_a} + \sqrt{- u_b} \right)^2}
\\
\frac{a+b}{2} + \frac{u_a + u_b}{b-a} t,\ &t_{{}_{b,1}} \leq t \leq t_{{}_{b,2}} := \frac{(b-a)^2}{- 2 ( u_a + u_b )},
\\
b - \sqrt{- 2 ( u_a + u_b ) t},\ &t \geq t_{{}_{b,2}},
\end{cases}
\\
\\
\gamma_{{}_{d}} (t) &:= d,\ t \geq 0.
\end{aligned}
\]
The velocity component $u$ can now be described as follows:

\[
\begin{aligned}
u(x,t) = \begin{cases}
0,\ &x \in \left( -\infty , \gamma_{{}_{a,1}} (t) \right) \cup \Big( \gamma_{{}_{a,2}} (t) , \gamma_{{}_{b}} (t) \Big) \cup \Big( b , \infty \Big),
\\
\frac{x-a}{t},\ &x \in \Big( \gamma_{{}_{a,1}} (t) , \gamma_{{}_{a,2}} (t) \Big),
\\
\frac{x-b}{t},\ &x \in \Big( \gamma_{{}_{b}} (t) , b \Big).
\end{cases}
\end{aligned} 
\] 
To describe the density component $\rho$, let us first consider the coordinates
\[
\begin{aligned}
\left( x^* , t^* \right) := \left( a + \frac{\left( b-a \right) \sqrt{u_a}}{\sqrt{u_a} + \sqrt{- u_b}} , \frac{(b-a)^2}{2 \left( \sqrt{u_a} + \sqrt{-u_b} \right)^2} \right).
\end{aligned}
\]
There are three subcases to be considered before we can proceed further.

First let us consider the subcase $c < x^*$. Define the curve
\[
\begin{aligned}
\gamma_{{}_{c}} (t) := \begin{cases}
c,\ &0 \leq t \leq t_{{}_{c,1}} := \frac{(c-a)^2}{2 u_a},
\\
a + \sqrt{2 u_a t},\ &t_{{}_{c,1}} \leq t \leq t_{{}_{c,2}} := t^*,
\\
\frac{a+b}{2} + \frac{u_a + u_b}{b-a} t,\ &t_{{}_{c,2}} \leq t \leq t_{{}_{c,3}} := \frac{(b-a)^2}{- 2 ( u_a + u_b )},
\\
b - \sqrt{- 2 ( u_a + u_b ) t},\ &t \geq t_{{}_{c,3}}.
\end{cases}
\end{aligned}
\]
Now any test function $\phi \in C_c^\infty \left( (-\infty,\infty) \times [0,\infty) \right)$ satisfies

\[
\begin{aligned}
	\langle R_x , \phi \rangle 
		&= - \langle R , \phi_x \rangle
	\\
		&= - \int_0^\infty \int_{\gamma_{{}_{c}} (t)}^{\gamma_d (t)} \rho_c\ \phi_x\ dx\ dt - \int_0^\infty \int_{\gamma_{{}_{s}} (t)}^{\infty} ( \rho_c + \rho_d )\ \phi_x\ dx\ dt
	\\
		&= \int_0^\infty \rho_c \left[ \phi \left( \gamma_{{}_{c}} (t) , t \right) - \phi \left( \gamma_{{}_{d}} (t) , t \right) \right] dt + \int_0^\infty ( \rho_c + \rho_d )\ \phi \left( \gamma_{{}_{d}} (t) , t \right)\ dt
	\\
		&= \langle \rho_c \left( \delta_{x = \gamma_{{}_{c}} (t)} - \delta_{x = \gamma_{{}_{d}} (t)} \right) + ( \rho_c + \rho_d )\ \delta_{x = \gamma_{{}_{d}} (t)} , \phi \rangle
	\\
		&= \langle \rho_c\ \delta_{x = \gamma_{{}_{c}} (t)} + \rho_d\ \delta_{x = \gamma_{{}_{d}} (t)} , \phi \rangle,
\end{aligned}
\]
so that $\rho$ is given by

\[
\rho = \rho_c\ \delta_{x = \gamma_{{}_{c}} (t)} + \rho_d\ \delta_{x = \gamma_{{}_{d}} (t)}.
\]
This solution can now be graphically represented as follows:

\begin{center}
\includegraphics[scale=0.5]{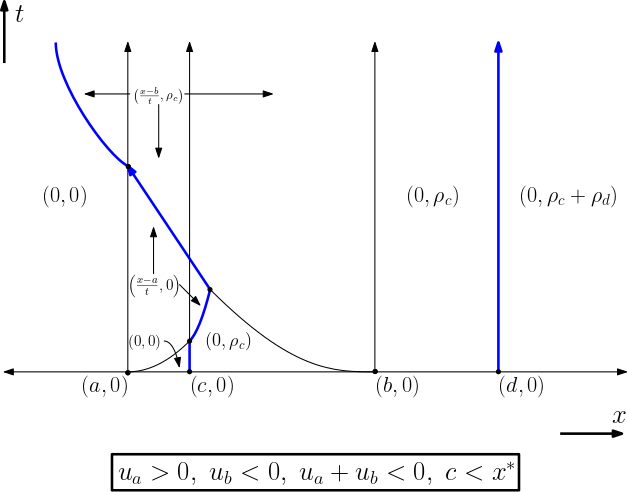}
\end{center}

Next we consider the second subcase $c = x^*$. Defining the curve $\gamma_{{}_{c}}$ by

\[
\begin{aligned}
\gamma_{{}_{c}} (t) := \begin{cases}
c,\ &0 \leq t \leq t_{{}_{c,1}} := t^*,
\\
\frac{a+b}{2} + \frac{u_a + u_b}{b-a} t,\ &t_{{}_{c,1}} \leq t \leq t_{{}_{c,2}} := \frac{(b-a)^2}{- 2 ( u_a + u_b )},
\\
b - \sqrt{- 2 ( u_a + u_b ) t},\ &t \geq t_{{}_{c,2}}
\end{cases}
\end{aligned}
\]
and proceeding with the computations as in the previous subcase, we will get

\[
\rho = \rho_c\ \delta_{x = \gamma_{{}_{c}} (t)} + \rho_d\ \delta_{x = \gamma_{{}_{d}} (t)}.
\] 
The graphical representation of our solution under this subcase is given below:

\begin{center}
\includegraphics[scale=0.56]{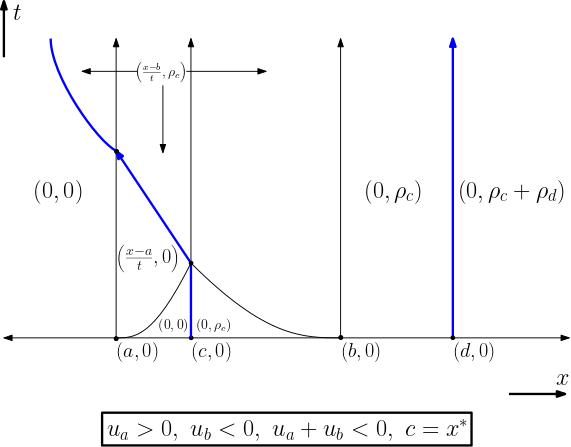}
\end{center}

The final subcase is $c > x^*$. We only define

\[
\begin{aligned}
\gamma_{{}_{c}} (t) := \begin{cases}
c,\ &0 \leq t \leq t_{{}_{c,1}} := \frac{(b-c)^2}{- 2 u_b},
\\
b - \sqrt{- 2 u_b t},\ &t_{{}_{c,1}} \leq t \leq t_{{}_{c,2}} := t^*,
\\
\frac{a+b}{2} + \frac{u_a + u_b}{b-a} t,\ &t_{{}_{c,2}} \leq t \leq t_{{}_{c,3}} := \frac{(b-a)^2}{- 2 ( u_a + u_b )},
\\
b - \sqrt{- 2 ( u_a + u_b ) t},\ &t \geq t_{{}_{c,3}}
\end{cases}
\end{aligned}
\]
and observe that the same computations as done for the previous two subcases lead us to the conclusion that

\[
\rho = \rho_c\ \delta_{x = \gamma_{{}_{c}} (t)} + \rho_d\ \delta_{x = \gamma_{{}_{d}} (t)}.
\] 
Under this subcase, our solution can be graphically represented as follows:

\begin{center}
\includegraphics[scale=0.52]{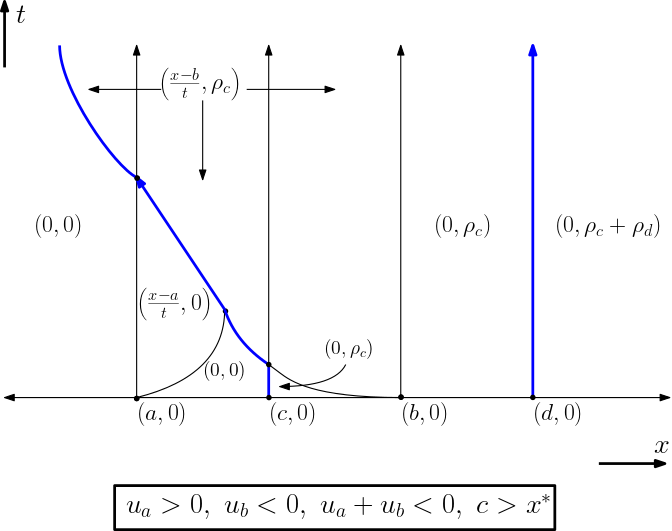}
\end{center}

\textbf{Case 5.} $u_a > 0$, $u_b < 0$, $u_a + u_b = 0$

The limit $\lim_{\epsilon \rightarrow 0} \left( u^\epsilon , R^\epsilon \right)$ is evaluated explicitly in different regions as follows:

\begin{itemize}

	\item $x < a$

\[
\begin{aligned}
\lim_{\epsilon \rightarrow 0} 
&\begin{pmatrix}
	\displaystyle{ \frac{\epsilon}{\sqrt{2 t \epsilon}} } \cdot \frac{\begin{aligned}
	\frac{1 - e^{- \frac{| u_a |}{\epsilon}}}{e^{A_\epsilon^2}} + \frac{e^{- \frac{| u_b |}{\epsilon}} - 1}{e^{B_\epsilon^2}}
\end{aligned}}{\begin{aligned}
	\sqrt{\pi} + \textnormal{erfc} \left( A_\epsilon \right) \left( e^{- \frac{| u_a |}{\epsilon}} - 1 \right) + \textnormal{erfc} \left( B_\epsilon \right) \left( 1 - e^{- \frac{| u_b |}{\epsilon}} \right)
\end{aligned} }
		\\
		\\
	\frac{
\begin{aligned}
\rho_c\ \textnormal{erfc} \left( B_\epsilon \right)\ \left( 1 - e^{- \frac{| u_b |}{\epsilon}} \right) \\
+ \rho_c\ \textnormal{erfc} \left( C_\epsilon \right) e^{- \frac{| u_b |}{\epsilon}} + \rho_d\ \textnormal{erfc} \left( D_\epsilon \right)
\end{aligned} 
}{ 
\begin{aligned}
\sqrt{\pi} + \textnormal{erfc} \left( A_\epsilon \right)\ \left( e^{- \frac{| u_a |}{\epsilon}} - 1 \right) + \textnormal{erfc} \left( B_\epsilon \right)\ \left( 1 - e^{- \frac{| u_b |}{\epsilon}} \right)
\end{aligned} }		
\end{pmatrix}^T		 		
\\
	= &\left( 0 , 0 \right).
\end{aligned}
\]

	\item $a < x < c$

\[
\begin{aligned}
\lim_{\epsilon \rightarrow 0} 
&\begin{pmatrix}
	\displaystyle{ \frac{x-a}{2t} } \cdot \frac{\begin{aligned}
	\frac{1 - e^{- \frac{| u_a |}{\epsilon}}}{A_\epsilon\ e^{A_\epsilon^2 - \frac{u_a}{\epsilon}}} + \frac{b-x}{x-a} \cdot \frac{e^{- \frac{| u_b |}{\epsilon}} - 1}{B_\epsilon\ e^{B_\epsilon^2 + \frac{u_b}{\epsilon}}}
\end{aligned}}{\begin{aligned}
	\sqrt{\pi} + \textnormal{erfc} \left( A_\epsilon \right) e^{\frac{| u_a |}{\epsilon}} \left( 1 - e^{- \frac{| u_a |}{\epsilon}} \right) + \textnormal{erfc} \left( B_\epsilon \right) e^{\frac{| u_b |}{\epsilon}} \left( 1 - e^{- \frac{| u_b |}{\epsilon}} \right)
\end{aligned} }
		\\
		\\
	\frac{
\begin{aligned}
\rho_c\ \textnormal{erfc} \left( B_\epsilon \right) e^{\frac{| u_b |}{\epsilon}} \left( 1 - e^{- \frac{| u_b |}{\epsilon}} \right) 
\\
+ \rho_c\ \textnormal{erfc} \left( C_\epsilon \right) + \rho_d\ \textnormal{erfc} \left( D_\epsilon \right) e^{\frac{| u_b |}{\epsilon}}
\end{aligned} 
}{ 
\begin{aligned}
\sqrt{\pi} + \textnormal{erfc} \left( A_\epsilon \right) e^{\frac{| u_a |}{\epsilon}} \left( 1 - e^{- \frac{| u_a |}{\epsilon}} \right) + \textnormal{erfc} \left( B_\epsilon \right) e^{\frac{| u_b |}{\epsilon}} \left( 1 - e^{- \frac{| u_b |}{\epsilon}} \right)
\end{aligned} }				
\end{pmatrix}^T		 		
\\
	= &\begin{cases}
		\left( 0 , 0 \right),\ &{ x \in \Big( p_a (t) , p_b (t) \Big) },
		\\
		\left( \frac{x-b}{t} , \rho_c \right),\ &{ x \in \Big( \max{ \big\{ p_a (t) , p_b (t) \big\} } , \infty \Big) \cup \Big( \max{ \big\{ p_b (t) , \frac{a+b}{2} \big\} } , p_a (t) \Big) },
		\\
		\left( \frac{x-a}{t} , 0 \right),\ &{ x \in \Big( -\infty , \min{ \big\{ p_a (t) , p_b (t) \big\} } \Big) \cup \Big( p_b (t) , \min{ \big\{ p_a (t) , \frac{a+b}{2} \big\} } \Big) }.
		\end{cases}
\end{aligned}
\]

	\item $c < x < b$

\[
\begin{aligned}
\lim_{\epsilon \rightarrow 0} 
&\begin{pmatrix}
	\displaystyle{ \frac{x-a}{2t} } \cdot \frac{\begin{aligned}
	\frac{1 - e^{- \frac{| u_a |}{\epsilon}}}{A_\epsilon\ e^{A_\epsilon^2 - \frac{u_a}{\epsilon}}} + \frac{b-x}{x-a} \cdot \frac{e^{- \frac{| u_b |}{\epsilon}} - 1}{B_\epsilon\ e^{B_\epsilon^2 + \frac{u_b}{\epsilon}}}
\end{aligned}}{\begin{aligned}
	\sqrt{\pi} + \textnormal{erfc} \left( A_\epsilon \right) e^{\frac{| u_a |}{\epsilon}} \left( 1 - e^{- \frac{| u_a |}{\epsilon}} \right) + \textnormal{erfc} \left( B_\epsilon \right) e^{\frac{| u_b |}{\epsilon}} \left( 1 - e^{- \frac{| u_b |}{\epsilon}} \right)
\end{aligned} }
		\\
		\\
	\frac{
\begin{aligned}
\sqrt{\pi}\ \rho_c + \rho_c\ \textnormal{erfc} \left( B_\epsilon \right) e^{\frac{| u_b |}{\epsilon}} \left( 1 - e^{- \frac{| u_b |}{\epsilon}} \right) 
\\
- \rho_c\ \textnormal{erfc} \left( C_\epsilon \right) + \rho_d\ \textnormal{erfc} \left( D_\epsilon \right) e^{\frac{| u_b |}{\epsilon}}
\end{aligned} 
}{ 
\begin{aligned}
\sqrt{\pi} + \textnormal{erfc} \left( A_\epsilon \right) e^{\frac{| u_a |}{\epsilon}} \left( 1 - e^{- \frac{| u_a |}{\epsilon}} \right) + \textnormal{erfc} \left( B_\epsilon \right) e^{\frac{| u_b |}{\epsilon}} \left( 1 - e^{- \frac{| u_b |}{\epsilon}} \right)
\end{aligned} }		
\end{pmatrix}^T		 		
\\
	= &\begin{cases}
		\left( 0 , \rho_c \right),\ &{ x \in \Big( p_a (t) , p_b (t) \Big) },
		\\
		\left( \frac{x-b}{t} , \rho_c \right),\ &{ x \in \Big( \max{ \big\{ p_a (t) , p_b (t) \big\} } , \infty \Big) \cup \Big( \max{ \big\{ p_b (t) , \frac{a+b}{2} \big\} } , p_a (t) \Big) },
		\\
		\left( \frac{x-a}{t} , 0 \right),\ &{ x \in \Big( -\infty , \min{ \big\{ p_a (t) , p_b (t) \big\} } \Big) \cup \Big( p_b (t) , \min{ \big\{ p_a (t) , \frac{a+b}{2} \big\} } \Big) }.
		\end{cases}
\end{aligned}
\]

	\item $b < x < d$

\[
\begin{aligned}
\lim_{\epsilon \rightarrow 0} 
&\begin{pmatrix}
	\displaystyle{ \frac{\epsilon}{\sqrt{2 t \epsilon}} } \cdot \frac{\begin{aligned}
	\frac{1 - e^{- \frac{| u_a |}{\epsilon}}}{e^{A_\epsilon^2}} + \frac{e^{- \frac{| u_b |}{\epsilon}} - 1}{e^{B_\epsilon^2}}
\end{aligned}}{\begin{aligned}
	\sqrt{\pi} + \textnormal{erfc} \left( A_\epsilon \right) \left( 1 - e^{- \frac{| u_a |}{\epsilon}} \right) + \textnormal{erfc} \left( B_\epsilon \right) \left( e^{- \frac{| u_b |}{\epsilon}} - 1 \right)
\end{aligned} }
		\\
		\\
	\frac{
\begin{aligned}
\sqrt{\pi}\ \rho_c + \rho_c\ \textnormal{erfc} \left( B_\epsilon \right) \left( e^{- \frac{| u_b |}{\epsilon}} - 1 \right) 
\\
- \rho_c\ \textnormal{erfc} \left( C_\epsilon \right) e^{- \frac{| u_b |}{\epsilon}} + \rho_d\ \textnormal{erfc} \left( D_\epsilon \right)
\end{aligned} 
}{ 
\begin{aligned}
\sqrt{\pi} + \textnormal{erfc} \left( A_\epsilon \right) \left( 1 - e^{- \frac{| u_a |}{\epsilon}} \right) + \textnormal{erfc} \left( B_\epsilon \right) \left( e^{- \frac{| u_b |}{\epsilon}} - 1 \right)
\end{aligned} }		
\end{pmatrix}^T		 		
\\
	= &\left( 0 , \rho_c \right).
\end{aligned}
\]

	\item $x > d$

\[
\begin{aligned}
\lim_{\epsilon \rightarrow 0} 
&\begin{pmatrix}
	\displaystyle{ \frac{\epsilon}{\sqrt{2 t \epsilon}} } \cdot \frac{\begin{aligned}
	\frac{1 - e^{- \frac{| u_a |}{\epsilon}}}{e^{A_\epsilon^2}} + \frac{e^{- \frac{| u_b |}{\epsilon}} - 1}{e^{B_\epsilon^2}}
\end{aligned}}{\begin{aligned}
	\sqrt{\pi} + \textnormal{erfc} \left( A_\epsilon \right) \left( 1 - e^{- \frac{| u_a |}{\epsilon}} \right) + \textnormal{erfc} \left( B_\epsilon \right) \left( e^{- \frac{| u_b |}{\epsilon}} - 1 \right)
\end{aligned} }
		\\
		\\
	\frac{
\begin{aligned}
\sqrt{\pi}\ \left( \rho_c + \rho_d \right) + \rho_c\ \textnormal{erfc} \left( B_\epsilon \right) \left( e^{- \frac{| u_b |}{\epsilon}} - 1 \right) 
\\
- \rho_c\ \textnormal{erfc} \left( C_\epsilon \right) e^{- \frac{| u_b |}{\epsilon}} - \rho_d\ \textnormal{erfc} \left( D_\epsilon \right)
\end{aligned} 
}{ 
\begin{aligned}
\sqrt{\pi} + \textnormal{erfc} \left( A_\epsilon \right) \left( 1 - e^{- \frac{| u_a |}{\epsilon}} \right) + \textnormal{erfc} \left( B_\epsilon \right) \left( e^{- \frac{| u_b |}{\epsilon}} - 1 \right)
\end{aligned} }		
\end{pmatrix}^T		 		
\\
	= &\left( 0 , \rho_c + \rho_d \right).
\end{aligned}
\]

\end{itemize}
To describe our solution, let us define the following curves:

\[
\begin{aligned}
\gamma_{{}_{a}} (t) &:= \begin{cases} 
a + \sqrt{2 u_a t},\ &0 \leq t \leq t_{{}_{a,1}} := \frac{(b-a)^2}{2 \left( \sqrt{u_a} + \sqrt{- u_b} \right)^2},
\\
\frac{a+b}{2},\ &t \geq t_{{}_{a,1}},
\end{cases}
\\
\\
\gamma_{{}_{b}} (t) &:= \begin{cases} 
b - \sqrt{- 2 u_b t},\ &0 \leq t \leq t_{{}_{b,1}} := \frac{(b-a)^2}{2 \left( \sqrt{u_a} + \sqrt{- u_b} \right)^2},
\\
\frac{a+b}{2},\ &t \geq t_{{}_{b,1}},
\end{cases}
\end{aligned}
\]
Then we have the following explicit representation of our velocity component $u$:

\[
u(x,t) = \begin{cases}
0,\ &x \in \Big( -\infty , a \Big) \cup \Big( \gamma_{{}_{a}} (t) , \gamma_{{}_{b}} (t) \Big) \cup \Big( b , \infty \Big),
\\
\frac{x-a}{t},\ &x \in \Big( a , \gamma_{{}_{a}} (t) \Big),
\\
\frac{x-b}{t},\ &x \in \Big( \gamma_{{}_{b}} (t) , b \Big).
\end{cases}
\]
To describe the density component $\rho$, we have to consider 3 subcases.

The first subcase is $c < \frac{a+b}{2}$. Define 
\[
\begin{aligned}
\gamma_{{}_{c}} (t) := \begin{cases}
c,\ &0 \leq t \leq t_{{}_{c,1}} := \frac{(c-a)^2}{2 u_a},
\\
a + \sqrt{2 u_a t},\ &t_{{}_{c,1}} \leq t \leq t_{{}_{c,2}} := \frac{(b-a)^2}{2 \left( \sqrt{u_a} + \sqrt{- u_b} \right)^2},
\\
\frac{a+b}{2},\ &t \geq t_{{}_{c,2}}.
\end{cases}
\end{aligned}
\]
Now any test function $\phi \in C_c^\infty \left( (-\infty,\infty) \times [0,\infty) \right)$ will satisfy

\[
\begin{aligned}
	\langle R_x , \phi \rangle 
		&= - \langle R , \phi_x \rangle
	\\
		&= - \int_0^\infty \int_{\gamma_c (t)}^{d} \rho_c\ \phi_x\ dx\ dt - \int_0^\infty \int_{d}^{\infty} ( \rho_c + \rho_d )\ \phi_x\ dx\ dt
	\\
		&= \int_0^\infty \rho_c \left[ \phi \left( \gamma_c (t) , t \right) - \phi \left( d , t \right) \right] dt + \int_0^\infty ( \rho_c + \rho_d )\ \phi \left( d , t \right)\ dt
	\\
		&= \langle \rho_c \left( \delta_{x = \gamma_c (t)} - \delta_{x = d} \right) + ( \rho_c + \rho_d )\ \delta_{x = d} , \phi \rangle
	\\
		&= \langle \rho_c\ \delta_{x = \gamma_c (t)} + \rho_d\ \delta_{x = d} , \phi \rangle.
\end{aligned}
\]
Therefore the density component $\rho$ can be described as

\[
\rho = \rho_c\ \delta_{x = \gamma_c (t)} + \rho_d\ \delta_{x = \gamma_{{}_{d}} (t)}.
\]
We can now graphically represent our solution under this subcase as follows:

\begin{center}
\includegraphics[scale=0.3]{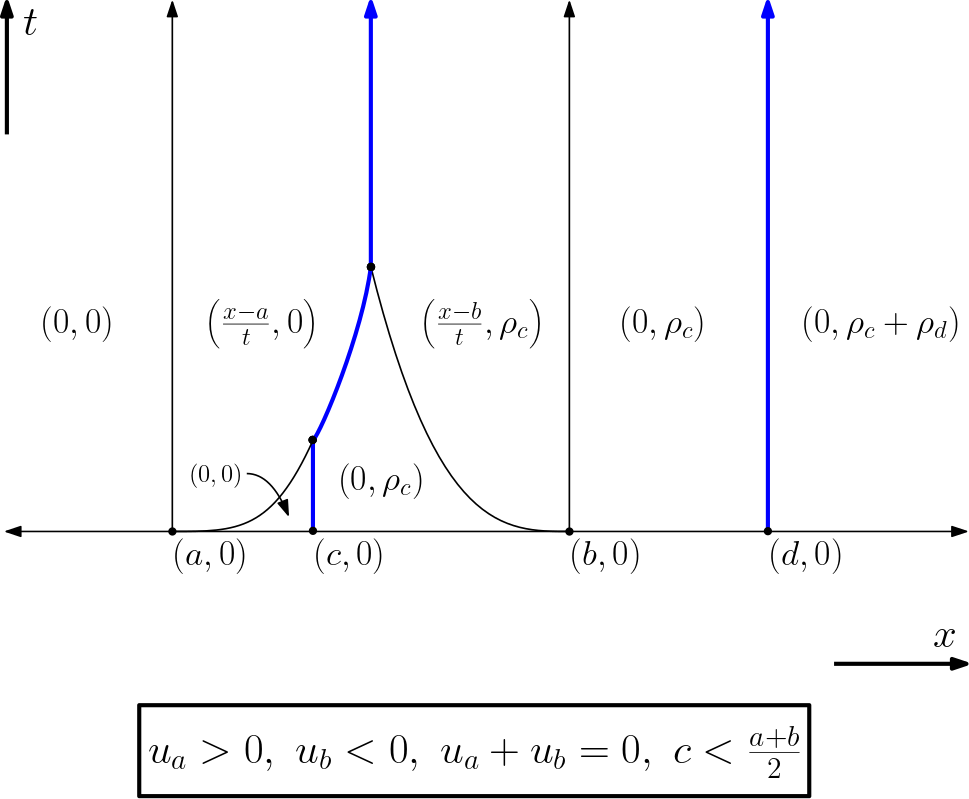}
\end{center}

Secondly, we consider the subcase $c = \frac{a+b}{2}$. Defining $\gamma_{{}_{c}} (t)$ to be $\frac{a+b}{2}$ for every $t \geq 0$, the density component $\rho$ will have exactly the same representation as before.

The new graphical representation of our solution under this subcase will be given as follows:

\begin{center}
\includegraphics[scale=0.4]{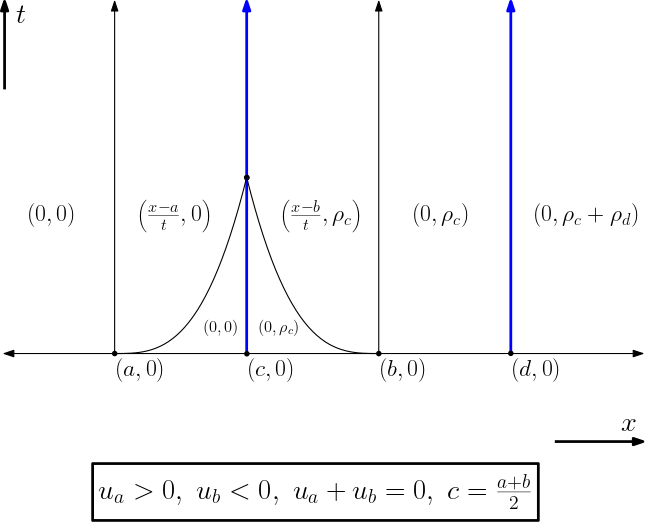}
\end{center}

The last subcase to be considered is $c > \frac{a+b}{2}$. We define the curve $\gamma_{{}_{c}}$ by

\[
\begin{aligned}
\gamma_{{}_{c}} (t) := \begin{cases}
c,\ &0 \leq t \leq t_{{}_{c,1}} := \frac{(b-c)^2}{- 2 u_b},
\\
b - \sqrt{- 2 u_b t},\ &t_{{}_{c,1}} \leq t \leq t_{{}_{c,2}} := \frac{(b-a)^2}{2 \left( \sqrt{u_a} + \sqrt{- u_b} \right)^2},
\\
\frac{a+b}{2},\ &t \geq t_{{}_{c,2}}.
\end{cases}
\end{aligned}
\]
With this definition, the density component $\rho$ has the same representation as in the previous subcases and the new graphical representation of our solution is given as follows:

\begin{center}
\includegraphics[scale=0.4]{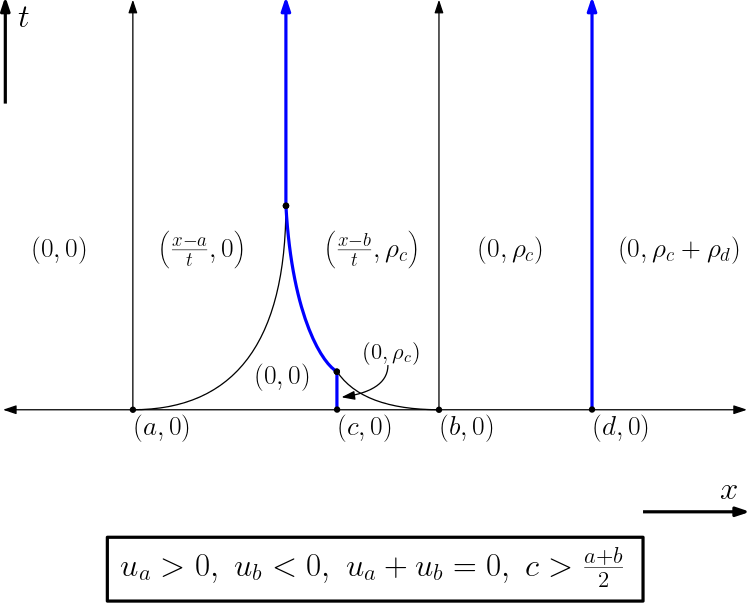}
\end{center}

\textbf{Case 6.} $u_a < 0$, $u_b < 0$

The limit $\lim_{\epsilon \rightarrow 0} \left( u^\epsilon , R^\epsilon \right)$ is evaluated explicitly in different regions as follows:

\begin{itemize}

	\item $x < a$
	
In this region, we introduce the curves

\begin{itemize}
	
	\item $p_a (s) := a - \sqrt{- 2 u_a s}$,
	
	\item $p_b (s) := b - \sqrt{- 2 \left( u_a + u_b \right) s}$,
	
	\item $l(s) := \frac{a+b}{2} + \frac{u_b}{b-a} s$

\end{itemize}	
defined for every $s \geq 0$. The required vanishing viscosity limit $\lim_{\epsilon \rightarrow 0} \left( u^\epsilon , R^\epsilon \right)$ is then obtained as follows:

\[
\begin{aligned}
\lim_{\epsilon \rightarrow 0} 
&\begin{pmatrix}
	\displaystyle{ \frac{a-x}{2t} } \cdot \frac{\begin{aligned}
	\frac{e^{- \frac{| u_a |}{\epsilon}} - 1}{A_\epsilon\ e^{A_\epsilon^2 + \frac{u_a}{\epsilon}}} + \frac{b-x}{a-x} \cdot \frac{e^{- \frac{| u_b |}{\epsilon}} - 1}{B_\epsilon\ e^{B_\epsilon^2 + \frac{u_a + u_b}{\epsilon}}}
\end{aligned}}{\begin{aligned}
	\sqrt{\pi} + \textnormal{erfc} \left( A_\epsilon \right) e^{\frac{| u_a |}{\epsilon}} \left( 1 - e^{- \frac{| u_a |}{\epsilon}} \right) + \textnormal{erfc} \left( B_\epsilon \right) e^{\frac{| u_a + u_b |}{\epsilon}} \left( 1 - e^{- \frac{| u_b |}{\epsilon}} \right)
\end{aligned} } 
		\\
		\\
	\frac{
\begin{aligned}
\rho_c\ \textnormal{erfc} \left( B_\epsilon \right) e^{\frac{| u_a + u_b |}{\epsilon}} \left( 1 - e^{- \frac{| u_b |}{\epsilon}} \right) \\
+ \rho_c\ \textnormal{erfc} \left( C_\epsilon \right) e^{\frac{| u_a |}{\epsilon}} + \rho_d\ \textnormal{erfc} \left( D_\epsilon \right) e^{\frac{| u_a + u_b |}{\epsilon}}
\end{aligned} 
}{ 
\begin{aligned}
\sqrt{\pi} + \textnormal{erfc} \left( A_\epsilon \right) e^{\frac{| u_a |}{\epsilon}} \left( 1 - e^{- \frac{| u_a |}{\epsilon}} \right) + \textnormal{erfc} \left( B_\epsilon \right) e^{\frac{| u_a + u_b |}{\epsilon}} \left( 1 - e^{- \frac{| u_b |}{\epsilon}} \right)
\end{aligned} }						
\end{pmatrix}^T		 		
\\
	= &\begin{cases}
		\left( 0 , 0 \right),\ &{ x < \min{ \big\{ p_a (t) , p_b (t) \big\} } },
		\\
		\left( \frac{x-b}{t} , \rho_c \right),\ &{ x \in \Big( p_b (t) , p_a (t) \Big) \cup \Big( \max{ \big\{ p_a (t) , p_b (t) , l(t) \big\} } , \infty \Big) },
		\\
		\left( \frac{x-a}{t} , 0 \right),\ &{ x \in \Big( p_a (t) , p_b (t) \Big) \cup \Big( \max{ \big\{ p_a (t) , p_b (t) \big\} } , l(t) \Big) }.
		\end{cases}
\end{aligned}
\]

	\item $a < x < c$

\[
\begin{aligned}
\lim_{\epsilon \rightarrow 0} 
&\begin{pmatrix}
	\displaystyle{ \frac{b-x}{2t} } \cdot \frac{\begin{aligned}
	\frac{x-a}{b-x} \cdot \frac{e^{- \frac{| u_a |}{\epsilon}} - 1}{A_\epsilon\ e^{A_\epsilon^2}} + \frac{e^{- \frac{| u_b |}{\epsilon}} - 1}{B_\epsilon\ e^{B_\epsilon^2 + \frac{u_b}{\epsilon}}}
\end{aligned}	}{\begin{aligned}
	\sqrt{\pi} + \textnormal{erfc} \left( A_\epsilon \right) \left( e^{- \frac{| u_a |}{\epsilon}} - 1 \right) + \textnormal{erfc} \left( B_\epsilon \right) e^{\frac{| u_b |}{\epsilon}} \left( 1 - e^{- \frac{| u_b |}{\epsilon}} \right)
\end{aligned}	 }
		\\
		\\
	\frac{
\begin{aligned}
\rho_c\ \textnormal{erfc} \left( B_\epsilon \right) e^{\frac{| u_b |}{\epsilon}} \left( 1 - e^{- \frac{| u_b |}{\epsilon}}  \right) 
\\
+ \rho_c\ \textnormal{erfc} \left( C_\epsilon \right) + \rho_d\ \textnormal{erfc} \left( D_\epsilon \right) e^{\frac{| u_b |}{\epsilon}}
\end{aligned} 
}{ 
\begin{aligned}
\sqrt{\pi} + \textnormal{erfc} \left( A_\epsilon \right) \left( e^{- \frac{| u_a |}{\epsilon}} - 1 \right) + \textnormal{erfc} \left( B_\epsilon \right) e^{\frac{| u_b |}{\epsilon}} \left( 1 - e^{- \frac{| u_b |}{\epsilon}} \right)
\end{aligned} }						
\end{pmatrix}^T		 		
\\
	= &\begin{cases}
		\left( 0 , 0 \right),\ &x < b - \sqrt{- 2 u_b t},
		\\
		\left( \frac{x-b}{t} , \rho_c \right),\ &x > b - \sqrt{- 2 u_b t}.
		\end{cases}
\end{aligned}
\]

	\item $c < x < b$

\[
\begin{aligned}
\lim_{\epsilon \rightarrow 0} 
&\begin{pmatrix}
	\displaystyle{ \frac{b-x}{2t} } \cdot \frac{\begin{aligned}
	\frac{x-a}{b-x} \cdot \frac{e^{- \frac{| u_a |}{\epsilon}} - 1}{A_\epsilon\ e^{A_\epsilon^2}} + \frac{e^{- \frac{| u_b |}{\epsilon}} - 1}{B_\epsilon\ e^{B_\epsilon^2 + \frac{u_b}{\epsilon}}}
\end{aligned}	}{\begin{aligned}
	\sqrt{\pi} + \textnormal{erfc} \left( A_\epsilon \right) \left( e^{- \frac{| u_a |}{\epsilon}} - 1 \right) + \textnormal{erfc} \left( B_\epsilon \right) e^{\frac{| u_b |}{\epsilon}} \left( 1 - e^{- \frac{| u_b |}{\epsilon}} \right)
\end{aligned}	 }
		\\
		\\
	\frac{
\begin{aligned}
\sqrt{\pi}\ \rho_c + \rho_c\ \textnormal{erfc} \left( B_\epsilon \right) e^{\frac{| u_b |}{\epsilon}} \left( 1 - e^{- \frac{| u_b |}{\epsilon}} \right) 
\\
- \rho_c\ \textnormal{erfc} \left( C_\epsilon \right) + \rho_d\ \textnormal{erfc} \left( D_\epsilon \right) e^{\frac{| u_b |}{\epsilon}}
\end{aligned} 
}{ 
\begin{aligned}
\sqrt{\pi} + \textnormal{erfc} \left( A_\epsilon \right) \left( e^{- \frac{| u_a |}{\epsilon}} - 1 \right) + \textnormal{erfc} \left( B_\epsilon \right) e^{\frac{| u_b |}{\epsilon}} \left( 1 - e^{- \frac{| u_b |}{\epsilon}} \right)
\end{aligned} }		
\end{pmatrix}^T		 		
\\
	= &\begin{cases}
		\left( 0 , \rho_c \right),\ &x < b - \sqrt{- 2 u_b t},
		\\
		\left( \frac{x-b}{t} , \rho_c \right),\ &x > b - \sqrt{- 2 u_b t}.
		\end{cases}	 
\end{aligned}
\]

	\item $b < x < d$

\[
\begin{aligned}
\lim_{\epsilon \rightarrow 0} 
&\begin{pmatrix}
	\displaystyle{ \frac{ \epsilon}{\sqrt{2 t \epsilon}} } \cdot \frac{\begin{aligned}
	\frac{e^{- \frac{| u_a + u_b |}{\epsilon}} - e^{- \frac{| u_b |}{\epsilon}}}{e^{A_\epsilon^2}} + \frac{e^{- \frac{| u_b |}{\epsilon}} - 1}{e^{B_\epsilon^2}}
\end{aligned}	}{\begin{aligned}
	\sqrt{\pi} + \textnormal{erfc} \left( A_\epsilon \right) \left( e^{- \frac{| u_a + u_b |}{\epsilon}} - e^{- \frac{| u_b |}{\epsilon}} \right) + \textnormal{erfc} \left( B_\epsilon \right) \left( e^{- \frac{| u_b |}{\epsilon}} - 1 \right)
\end{aligned}	 }
		\\
		\\
	\frac{
\begin{aligned}
\sqrt{\pi}\ \rho_c + \rho_c\ \textnormal{erfc} \left( B_\epsilon \right) \left( e^{- \frac{| u_b |}{\epsilon}} - 1 \right) 
\\
- \rho_c\ \textnormal{erfc} \left( C_\epsilon \right) e^{- \frac{| u_b |}{\epsilon}} + \rho_d\ \textnormal{erfc} \left( D_\epsilon \right)
\end{aligned} 
}{ 
\begin{aligned}
\sqrt{\pi} + \textnormal{erfc} \left( A_\epsilon \right) \left( e^{- \frac{| u_a + u_b |}{\epsilon}} - e^{- \frac{| u_b |}{\epsilon}} \right) + \textnormal{erfc} \left( B_\epsilon \right) \left(  e^{- \frac{| u_b |}{\epsilon}} - 1 \right)
\end{aligned} }		
\end{pmatrix}^T		 		
\\
	= &\left( 0 , \rho_c \right).
\end{aligned}
\]

	\item $x > d$

\[
\begin{aligned}
\lim_{\epsilon \rightarrow 0} 
&\begin{pmatrix}
	\displaystyle{ \frac{ \epsilon}{\sqrt{2 t \epsilon}} } \cdot \frac{\begin{aligned}
	\frac{e^{- \frac{| u_a + u_b |}{\epsilon}} - e^{- \frac{| u_b |}{\epsilon}}}{e^{A_\epsilon^2}} + \frac{e^{- \frac{| u_b |}{\epsilon}} - 1}{e^{B_\epsilon^2}}
\end{aligned}	}{\begin{aligned}
	\sqrt{\pi} + \textnormal{erfc} \left( A_\epsilon \right) \left( e^{- \frac{| u_a + u_b |}{\epsilon}} - e^{- \frac{| u_b |}{\epsilon}} \right) + \textnormal{erfc} \left( B_\epsilon \right) \left( e^{- \frac{| u_b |}{\epsilon}} - 1 \right)
\end{aligned}	 }
		\\
		\\
	\frac{
\begin{aligned}
\sqrt{\pi}\ \left( \rho_c + \rho_d \right) + \rho_c\ \textnormal{erfc} \left( B_\epsilon \right) \left( e^{- \frac{| u_b |}{\epsilon}} - 1 \right) 
\\
- \rho_c\ \textnormal{erfc} \left( C_\epsilon \right) e^{- \frac{| u_b |}{\epsilon}} - \rho_d\ \textnormal{erfc} \left( D_\epsilon \right)
\end{aligned} 
}{ 
\begin{aligned}
\sqrt{\pi} + \textnormal{erfc} \left( A_\epsilon \right) \left(  e^{- \frac{| u_a + u_b |}{\epsilon}} -  e^{- \frac{| u_b |}{\epsilon}} \right) + \textnormal{erfc} \left( B_\epsilon \right) \left(  e^{- \frac{| u_b |}{\epsilon}} - 1 \right)
\end{aligned} }		
\end{pmatrix}^T		 		
\\
	= &\left( 0 , \rho_c + \rho_d \right).
\end{aligned}
\]

\end{itemize}
To describe our solution, let us introduce the following curves:
\[
\begin{aligned}
\gamma_{{}_{a,1}} (t)
	&:= \begin{cases}
a - \sqrt{- 2 u_a t},\ &0 \leq t \leq t_{a,1} := \frac{(b-a)^2}{2 \left( \sqrt{- ( u_a + u_b )} - \sqrt{- u_a} \right)^2},
	\\
b - \sqrt{- 2 ( u_a + u_b ) t},\ &t \geq t_{a,1},
\end{cases}	
	\\
	\\
\gamma_{{}_{a,2}} (t)
	&:= \begin{cases}
a,\ &0 \leq t \leq t_{a,2} := \frac{(b-a)^2}{- 2 u_b},
	\\
\frac{a+b}{2} + \frac{u_b}{b-a} t,\ &t_{a,2} \leq t \leq t_{a,3} := \frac{(b-a)^2}{2 \left( \sqrt{- ( u_a + u_b )} - \sqrt{- u_a} \right)^2},
	\\
b - \sqrt{- 2 ( u_a + u_b ) t},\ &t \geq t_{a,3},
\end{cases}	
	\\
	\\
\gamma_{{}_b} (t)
	&:= \begin{cases}
b - \sqrt{- 2 u_b t},\ &0 \leq t \leq t_{b,1} := \frac{(b-a)^2}{- 2 u_b},
	\\
\frac{a+b}{2} + \frac{u_b}{b-a} t,\ &t_{b,1} \leq t \leq t_{b,2} := \frac{(b-a)^2}{2 \left( \sqrt{- ( u_a + u_b )} - \sqrt{- u_a} \right)^2},
	\\
b - \sqrt{- 2 ( u_a + u_b ) t},\ &t \geq t_{b,2},
\end{cases}	
	\\
	\\
\gamma_{{}_c} (t)
	&:= \begin{cases}
c,\ &0 \leq t \leq t_{c,1} := \frac{(b-c)^2}{- 2 u_b},	
	\\
b - \sqrt{- 2 u_b t},\ &t_{c,1} \leq t \leq t_{c,2} := \frac{(b-a)^2}{- 2 u_b},
	\\
\frac{a+b}{2} + \frac{u_b}{b-a} t,\ &t_{c,2} \leq t \leq t_{c,3} := \frac{(b-a)^2}{2 \left( \sqrt{- ( u_a + u_b )} - \sqrt{- u_a} \right)^2},
	\\
b - \sqrt{- 2 ( u_a + u_b ) t},\ &t \geq t_{c,3},
\end{cases}	
\\
\\
\gamma_{{}_{d}} (t) &:= d,\ t \geq 0.
\end{aligned}
\]
The velocity component $u$ can now be described as follows:

\[
\begin{aligned}
u(x,t) = \begin{cases}
0,\ &x \in \Big( -\infty , \gamma_{{}_{a,1}} (t) \Big) \cup \Big( \gamma_{{}_{a,2}} (t) , \gamma_{{}_b} (t) \Big) \cup \Big( b , \infty \Big),
\\
\frac{x-a}{t},\ &x \in \Big( \gamma_{{}_{a,1}} (t) , \gamma_{{}_{a,2}} (t) \Big),
\\
\frac{x-b}{t},\ & x \in \Big( \gamma_{{}_b} (t) , b \Big).
\end{cases}
\end{aligned}
\]
Finally, for getting $\rho$, we fix $\phi \in C_c^\infty \left( (-\infty,\infty) \times [0,\infty) \right)$ and compute

\[
\begin{aligned}
	\langle R_x , \phi \rangle 
		&= - \langle R , \phi_x \rangle
	\\
		&= - \int_0^\infty \int_{\gamma_{{}_c} (t)}^{d} \rho_c\ \phi_x\ dx\ dt - \int_0^\infty \int_{d}^{\infty} ( \rho_c + \rho_d )\ \phi_x\ dx\ dt
	\\
		&= \int_0^\infty \rho_c \left[ \phi \left( \gamma_{{}_c} (t) , t \right) - \phi \left( d , t \right) \right] dt + \int_0^\infty ( \rho_c + \rho_d )\ \phi \left( d , t \right)\ dt
	\\
		&= \langle \rho_c \left( \delta_{x = \gamma_{{}_c} (t)} - \delta_{x = d} \right) + ( \rho_c + \rho_d )\ \delta_{x = d} , \phi \rangle
	\\
		&= \langle \rho_c\ \delta_{x = \gamma_{{}_c} (t)} + \rho_d\ \delta_{x = d} , \phi \rangle.
\end{aligned}
\]
This implies that $\rho$ is given by

\[
\rho = \rho_c\ \delta_{x = \gamma_{{}_c} (t)} + \rho_d\ \delta_{x = \gamma_{{}_{d}} (t)},
\]
where $\gamma_{{}_{d}} : s \longmapsto d$ for every $s \geq 0$. This solution has the following graphical representation:

\begin{center}
\includegraphics[scale=0.4]{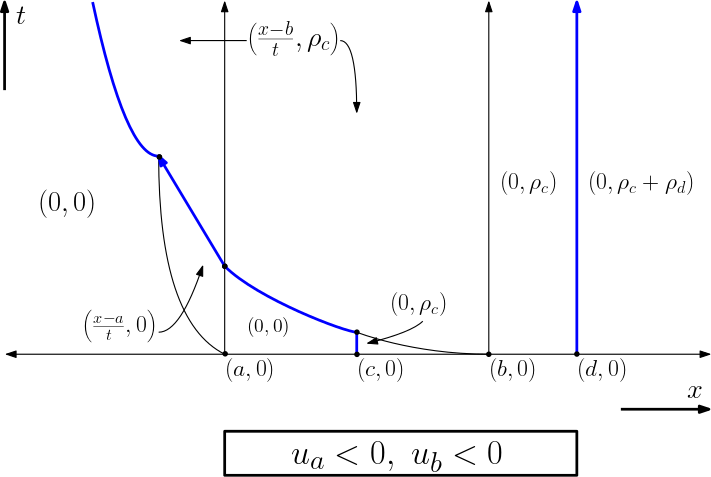}
\end{center}

\end{proof}

\section*{Appendix}

In this section, we derive the various properties of the function erfc used throughout this paper. Furthermore, we justify the computations leading to the explicit representations of $u^\epsilon$ and $R^\epsilon$ under the prescribed initial data.

\begin{enumerate}

	\item We prove the following four properties of erfc:

\begin{enumerate}

	\item For every $z \in \textbf{R}^1$, $\textnormal{erfc} \left( z \right) + \textnormal{erfc} \left( -z \right) = \sqrt{\pi}$.
	
	\item $\lim_{z \rightarrow \infty} \textnormal{erfc} \left( z \right) = 0$.
	
	\item $\textnormal{erfc} \left( z \right) = \left( \frac{1}{2 z} - \frac{1}{4 z^3} + o \left( \frac{1}{z^3} \right) \right)\ e^{- z^2} \textnormal{ as } z \rightarrow \infty$.
	
	\item $\lim_{z \rightarrow \infty}\ z\ \textnormal{erfc} \left( z \right)\ e^{z^2} = \frac{1}{2}$.

\end{enumerate}
The first property can be verified as follows:

\[
\begin{aligned}
\textnormal{erfc} \left( -z \right) = \int_{-z}^{\infty} e^{-y^2}\ dy &= \int_{-\infty}^{\infty} e^{-y^2}\ dy - \int_{-\infty}^{-z} e^{-y^2}\ dy
\\
&= \sqrt{\pi} - \int_{z}^{\infty} e^{-\eta^2}\ d\eta
\\
&= \sqrt{\pi} - \textnormal{erfc} \left( z \right),
\end{aligned}
\]
whereas the second property follows from the definition of erfc:

\[
\lim_{z \rightarrow \infty} \textnormal{erfc} \left( z \right) = \lim_{z \rightarrow \infty} \int_{z}^{\infty} e^{-s^2}\ ds = 0.
\]
We now verify the third and fourth properties. For any $z > 1$, we can integrate by parts to get

\[
\begin{aligned}
\textnormal{erfc} \left( z \right) 
		&= \int_{z}^{\infty}\ \left( - \frac{1}{2t} \right)\ \frac{d}{dt} \left( e^{-t^2} \right)\ dt
\\
		&= \frac{1}{2z}\ e^{-z^2} + \int_{z}^{\infty}\ \frac{1}{4 t^3}\ \frac{d}{dt} \left( e^{-t^2} \right)\ dt
\\
		&= \left( \frac{1}{2z} - \frac{1}{4 z^3} \right)\ e^{- z^2} + \int_{z}^{\infty}\ \frac{3}{4 t^4}\ e^{-t^2}\ dt		,
\end{aligned}
\]
so that

\[
\begin{aligned}
\left| z^3\ \left[ e^{z^2}\ \textnormal{erfc} \left( z \right) - \left( \frac{1}{2z} - \frac{1}{4 z^3} \right) \right] \right| 
		&\leq - \frac{3}{8 z^2}\ \int_{z}^{\infty}\ \frac{d}{dt}\ \left( e^{z^2 - t^2} \right)\ dt 
\\
		&\leq \frac{3}{8 z^2}.	
\end{aligned}
\]
Since $\lim_{z \rightarrow \infty} \frac{3}{8 z^2} = 0$, this proves our claim. This property also implies that

\[
z^2\ \left| z\ \textnormal{erfc} \left( z \right)\ e^{z^2} - \frac{1}{2} \right| \rightarrow \frac{1}{4} \textnormal{ as } z \rightarrow \infty,
\]
and hence, we can conclude that $\lim_{z \rightarrow \infty}\ z\ \textnormal{erfc} \left( z \right)\ e^{z^2} = \frac{1}{2}$.

\item The underlying computations are justified for the case $u_a < 0$, $u_b > 0$ in the region $x<a$. Similar computation methods can be used in the other regions under the remaining cases.

\begin{itemize}

	\item $\displaystyle{ \lim_{\epsilon \rightarrow 0}} e^{- \frac{| u_a |}{\epsilon}} = \lim_{\epsilon \rightarrow 0} e^{- \frac{| u_b |}{\epsilon}} = 0$

	\item $\displaystyle{ \lim_{\epsilon \rightarrow 0}} A_\epsilon\ e^{{A_\epsilon}^2 + \frac{u_a}{\epsilon}} \cdot \textnormal{erfc} \left( A_\epsilon \right)  e^{\frac{| u_a |}{\epsilon}} = \lim_{\epsilon \rightarrow 0} A_\epsilon\ \textnormal{erfc} \left( A_\epsilon \right) e^{A_\epsilon^2} = \frac{1}{2}$

	\item $\displaystyle{ \lim_{\epsilon \rightarrow 0}} \frac{A_\epsilon\ e^{{A_\epsilon}^2 + \frac{u_a}{\epsilon}}}{B_\epsilon\ e^{{B_\epsilon}^2 + \frac{u_a}{\epsilon}}} = \lim_{\epsilon \rightarrow 0} \frac{a-x}{b-x} \cdot e^{\frac{(a-x)^2 - (b-x)^2}{2 t \epsilon}} = 0$
	
	\item $\displaystyle{ \lim_{\epsilon \rightarrow 0}} A_\epsilon\ e^{{A_\epsilon}^2 + \frac{u_a}{\epsilon}} \textnormal{erfc} \left( B_\epsilon \right)  e^{\frac{| u_a |}{\epsilon}}  = \lim_{\epsilon \rightarrow 0} \frac{\left( a-x \right) f \left( B_\epsilon \right)}{\left( b-x \right) e^{\frac{(b-x)^2 - (a-x)^2}{2 t \epsilon}}}  = 0$
	
The last limit relation will also imply that

\[
\begin{aligned}
\lim_{\epsilon \rightarrow 0} A_\epsilon\ e^{{A_\epsilon}^2 + \frac{u_a}{\epsilon}} \cdot \textnormal{erfc} \left( C_\epsilon \right)  e^{\frac{| u_a |}{\epsilon}} 
	&= \lim_{\epsilon \rightarrow 0} \frac{a-x}{c-x} \cdot f(B_\epsilon) \cdot e^{\frac{(a-x)^2 - (c-x)^2}{2 t \epsilon}}
	\\
	&= 0;
\\
\\
\lim_{\epsilon \rightarrow 0} A_\epsilon\ e^{{A_\epsilon}^2 + \frac{u_a}{\epsilon}} \cdot \textnormal{erfc} \left( D_\epsilon \right)  e^{\frac{| u_a |}{\epsilon}} 
	&= \lim_{\epsilon \rightarrow 0} \frac{a-x}{d-x} \cdot f(B_\epsilon) \cdot e^{\frac{(a-x)^2 - (d-x)^2}{2 t \epsilon}}
	\\
	&= 0.			
\end{aligned}
\]	

	\item \textbf{Evaluation of ${\displaystyle{ \lim_{\epsilon \rightarrow 0}} } A_\epsilon\ e^{{A_\epsilon}^2 + \frac{u_a}{\epsilon}}$}:

\[
\begin{aligned}
\lim_{\epsilon \rightarrow 0} A_\epsilon\ e^{{A_\epsilon}^2 + \frac{u_a}{\epsilon}} 
	&= \lim_{\epsilon \rightarrow 0} \frac{a-x}{\sqrt{2 t \epsilon}}\ e^{\frac{(a-x)^2 + 2 u_a t}{2 t \epsilon}}
\\
	&= \begin{cases}
	\infty,\ &{x < a - \sqrt{- 2 u_a t}},
	\\
	0,\ &{x > a - \sqrt{- 2 u_a t}}.	
	\end{cases}
\end{aligned}
\]	

	\item \textbf{Evaluation of $\displaystyle{ \lim_{\epsilon \rightarrow 0}} \textnormal{erfc} \left( A_\epsilon \right) e^{\frac{| u_a |}{\epsilon}}$}:

\[
\begin{aligned}
\lim_{\epsilon \rightarrow 0} \textnormal{erfc} \left( A_\epsilon \right) e^{\frac{| u_a |}{\epsilon}} 
	&= \lim_{\epsilon \rightarrow 0} \frac{f(A_\epsilon)}{A_\epsilon\ e^{{A_\epsilon}^2 + \frac{u_a}{\epsilon}}} 
	\\
	&= \begin{cases}
	0,\ &{x < a - \sqrt{- 2 u_a t}},
	\\
	\infty,\ &{x > a - \sqrt{- 2 u_a t}}.
	\end{cases}
\end{aligned}
\]

\end{itemize}

\end{enumerate}

\section*{Acknowledgements}

The research was supported by the project "Basic research in physics and multidisciplinary sciences" (Identification \# RIN4001). The author is thankful to Professor K. T. Joseph for suggesting this problem. He also acknowledges the cooperation and support of Professor Anupam Pal Choudhury, Professor Agnid Banerjee and Professor K. T. Joseph during the preparation of this paper.

\end{document}